\documentclass{article}
\usepackage{graphicx}
\usepackage{subcaption}
\usepackage{transparent}
\usepackage{biblatex}

\usepackage{tikz-cd}
\usepackage{tikz}
\usepackage{bbm}

\usepackage{dsfont}
\usepackage{url}
\usepackage{hyperref}
\usepackage{amsmath}
\usepackage{amssymb}
\usepackage{mathrsfs}
\usepackage{braket}
\usepackage{mathtools}
\usepackage{booktabs}
\usepackage{caption}
\usepackage{cancel}
\usepackage{epstopdf}
\usepackage{comment}
\usepackage{appendix}
\DeclarePairedDelimiter{\abs}{\lvert}{\rvert}
\DeclarePairedDelimiter{\norm}{\lVert}{\rVert}
\newcommand{\numberset}{\mathbb}

\newcommand{\R}{\numberset{R}}

\newcommand{\Z}{\numberset{Z}}
\newcommand{\C}{\numberset{C}}

\newcommand{\sphere}{\mathbb{S}}

\newcommand{\Ham}{\text{Ham}}
\newcommand{\Symp}{\text{Symp}}
\newcommand{\Sym}{\text{Sym}}
\newcommand{\Diff}{\text{Diff}}

\newcommand{\D}{\mathbb{D}}

\usepackage{amsthm}
\newtheorem{thm}{Theorem}[section]

\newtheorem{lem}[thm]{Lemma}
\newtheorem{prp}[thm]{Proposition}
\newtheorem{cor}[thm]{Corollary}
\newtheorem{defn}[thm]{Definition}

\newtheorem{rem}[thm]{Remark}

\makeatletter
\newcommand*{\@old@slash}{}\let\@old@slash\slash
\def\slash{\relax\ifmmode\delimiter"502F30E\mathopen{}\else\@old@slash\fi}
\makeatother

\makeatletter
\def\bign#1{\mathclose{\hbox{$\left#1\vbox to8.5\p@{}\right.\n@space$}}\mathopen{}}
\def\Bign#1{\mathclose{\hbox{$\left#1\vbox to11.5\p@{}\right.\n@space$}}\mathopen{}}
\def\biggn#1{\mathclose{\hbox{$\left#1\vbox to14.5\p@{}\right.\n@space$}}\mathopen{}}
\def\Biggn#1{\mathclose{\hbox{$\left#1\vbox to17.5\p@{}\right.\n@space$}}\mathopen{}}
\makeatother

\title{Hofer Energy and Link Preserving Diffeomorphisms in Higher Genus}
\author{Francesco Morabito, Ibrahim Trifa}

\begin{document}

\maketitle
\begin{abstract}
    Given a pre-monotone Lagrangian link, we obtain Hofer energy estimates for Hamiltonian diffeomorphisms preserving it. Such estimates depend on the braid type of the Hamiltonian diffeomorphism only, and the natural language to talk about this phenomenon is provided by a family of norms on braid groups for surfaces with boundary.
    This generalises the results obtained by the first author to higher genus surfaces with boundary.
\end{abstract}

\tableofcontents
\section{Introduction}
Let $\Sigma_{g, p}$ be a compact surface with genus $g\geq 0$ and $p\geq 1$ boundary components (or punctures: we will use the two notions interchangeably since the content of the present paper is not affected by the difference). If $\Sigma_{g, p}$ is oriented, which will be the case throughout this paper, it is symplectic: there exists an area form $\omega$. We renormalise $\omega$ in such a way that $\int_{\Sigma_{g, p}}\omega=1$. Given a smooth function $H\in\mathcal{C}_0^\infty(\sphere^1_t\times\Sigma_{g, p}; \R)$ compactly supported in the interior of the surface for all times $t$, one may define its Hamiltonian vector field $X_{H_t}$ by\begin{equation*}
    \omega(X_{H_t}, \cdot)=-dH_t
\end{equation*}
The function $H$ one uses to define the Hamiltonian vector field is itself called Hamiltonian. The set of time 1-flows of such vector fields is a group, bearing the name of group of compactly-supported Hamiltonian diffeomorphisms of $\Sigma_{g, p}$. We denote it by $\Ham_c(\Sigma_{g, p})$. If $H$ is a Hamiltonian, we write $\phi^1_H$ for its time 1-map.

The group $\Ham_c(\Sigma_{g, p})$ comes with a natural metric on it, defined by Helmut Hofer in \cite{hof90}, as follows. For any smooth Hamiltonian $H$ on $\Sigma_{g, p}$ we define its oscillation
\begin{equation*}
    \norm{H}_{(1, \infty)}:=\int_0^1\left\{\max_{x\in M}H_t(x)-\min_{x\in M}H_t(x)\right\} dt
\end{equation*}
Let $\varphi\in \Ham_c(\Sigma_{g, p})$: its Hofer norm is defined then as
\begin{equation*}
    \norm{\varphi}:=\inf_{H\vert \varphi=\phi^1_H}\norm{H}_{(1, \infty)}
\end{equation*}

It is by now classical that $\norm{\cdot}$ defines a norm (see \cite{hof90}, \cite{lalMcD95},\cite{pol93}), and its very definition implies that $\norm{\cdot}$ is invariant under conjugation by symplectic diffeomorphisms (diffeomorphisms $\psi$ of the manifold such that $\psi^*\omega=\omega$). We define the Hofer metric extending $\norm{\cdot}$ by bi-invariance:
\begin{equation*}
    \text{If }\varphi, \varphi'\in\Ham_c(\Sigma_{g, p}),\,\,\, d_H(\varphi, \varphi'):=\norm{\varphi(\varphi')^{-1}}
\end{equation*}
It is rather complicated to provide examples of Hamiltonian diffeomorphisms with large Hofer norms. In principle, one expects Hofer energy to measure how ``large'' sets are moved around the symplectic manifold, and how complex their motion is. The first author of the present article constructed in \cite{M} a class of compactly supported Hamiltonian diffeomorphisms of the disc whose Hofer energy satisfies some simple estimates from below. The goal is now to extend these estimates to higher genus symplectic surfaces.
\subsection{Braid group on surfaces with genus and boundary}
Geometrically speaking, the braid group can be seen as the fundamental group of the configuration space of the surface $\Sigma_{g, p}$. Let us fix $k\geq 1$ distinct points on $\Sigma_{g, p}$, $P_1,\dots, P_k$, and consider the cartesian product $\left(\Sigma_{g, p}\right)^{k}$. The permutation group on $k$ letters, $\mathfrak{S}_k$
, acts on $\left(\Sigma_{g, p}\right)^{k}$, let $\Sym^k\left(\Sigma_{g, p}\right)$ be the quotient. Denote by $\Delta$ the image of the points in $\left(\Sigma_{g, p}\right)^{k}$ for which at least two coordinates coincide. We define $\mathcal{B}_{k,g, p}$ as the fundamental group of the $k$-th (unordered) configuration space of $\Sigma_{g, p}$, with base point $\{P_1,...,P_k\}$:
\begin{equation*}
    \mathcal{B}_{k, g, p}:=\pi_1\left(\Sym^k\left(\Sigma_{g, p}\right)\setminus \Delta\right)
\end{equation*}
If $g=0, p=1$ (i.e. $\Sigma_{g, p}$ is the disc) we may describe this group via a presentation containing $k-1$ generators which obey the so-called braid relations:\begin{equation*}
\mathcal{B}_{k, 0, 1} = \left \langle \sigma_1, \dots, \sigma_{k-1} \middle\vert
\begin{aligned}
 \sigma_i\sigma_j & = \sigma_j\sigma_i : |i-j| > 1 \\
 \sigma_i\sigma_{i+1} \sigma_i & = \sigma_{i+1} \sigma_i\sigma_{i+1} : 1 \leq i\leq k-2
\end{aligned}
\right\rangle
\end{equation*}
In such a case we have a remarkable homomorphism $\mathcal{B}_{k, 0, 1}\rightarrow\Z$, the linking number, obtained by sending every $\sigma_i$ to 1: this map will be denoted ``$\mathrm{lk}$'' throughout the paper. Note that each $\sigma_i$ corresponds geometrically to swapping $P_i$ and $P_{i+1}$ by a half-twist.

If however $g\geq 1, p\geq 1$, the presentation gets substantially more complicated. We are going to report here the results we need from \cite{bel04}. Let us fix $k$ distinct base points on $\Sigma_{g, p}$. There exist then four families of generators (for a picture, see Figure \ref{fig:generators}):
\begin{itemize}
    \item $\sigma_1, \dots, \sigma_{k-1}$: they correspond to the generators of $\mathcal{B}_{k, 0, 1}$, i.e. to half-twist swapping two of the base points, in such a way that the images of the paths are contained in a disc on the surface;
    \item $a_1, \dots, a_g$, $b_1, \dots, b_g$: they are obtained by homologically independent loops based at the first base point, $P_1$ (the other base points are instead fixed throughout the path). One way to describe them is seeing $\Sigma_{g, p}$ as the connected sum of $g$ tori with $p$ punctures, so that the loops $a_i$ and $b_i$ represent the generators of the homology of the $i$-th torus;
    \item $z_1, \dots, z_{p-1}$: they correspond to loops based at $P_1$ and winding around the $i-th$ puncture exactly once (here as well the other base points are fixed). Remark that there are $p-1$ generators, but $p$ punctures. A loop around the last puncture may in fact be written as composition of the others.
\end{itemize}
The relations in this group are rather complex, and we are not going to report them here to keep the presentation lean. The interested reader will find a detailed account in \cite[Theorem 1.1]{bel04}.
\subsection{The main result}
Let $\underline{L}=L_1,\dots, L_{k+g}$ be a set of $k+g$ disjoint, embedded circles on $\Sigma_{g, p}$: we call $\underline{L}$ a Lagrangian link.
\begin{defn}\label{defn:pre-monLink}
    We say that $\underline{L}$ is premonotone if the following conditions are satisfied:
\begin{itemize}
\item[\textit{i})] exactly $g$ of the circles in $\underline{L}$ are non contractible;
\item[\textit{ii})] $\Sigma_{g, p}\setminus \bigcup_{i=1}^{k+g}L_i$ is a disjoint union of $k$ discs $(B_j)_{j=1,\dots, k}$ and a pair of pants $B_{k+1}$ with $p+k+2g-1$ legs, i.e. a genus-0 surface with $p+k+2g$ ends; 
\item[\textit{iii})] There exist $\lambda\in \left(\frac{1}{k+1}, \frac{1}{k}\right)$, such that $\lambda=\int_{B_j}\omega$ for $j=1,\dots, k$.
\end{itemize}
\end{defn}
\begin{rem}
    The choice of the adjective ``premonotone'' is motivated by the fact that such Lagrangian configurations define monotone Lagrangian tori in symmetric products after the surface with boundary is embedded in a closed surface with same genus. The monotonicity of this kind of tori does not follow from previous results in \cite{CGHMSS22} and \cite{M}, and it is an element of novelty which appears to be one of the main ingredients for the success of the proof.
\end{rem}

Given a premonotone Lagrangian link, we denote by $\Ham_{\underline{L}}(\Sigma_{g, p})$ the subgroup of $\Ham_c(\Sigma_{g, p})$ stabilising the Lagrangian link as a set:
\begin{equation*}
\Ham_{\underline{L}}(\Sigma_{g, p}):=\Set{\varphi\in \Ham_c(\Sigma_{g, p})\vert \exists\sigma\in\mathfrak{S}_{k+g}, \varphi(L_i)=L_{\sigma(i)}}
\end{equation*}
To a diffeomorphism in $\Ham_{\underline{L}}(\Sigma_{g, p})$ we may associate an element of the braid group $\mathcal{B}_{k, g, p}$. This function is defined the following way: choose one base point per contractible circle in $\underline{L}$, denote it by $P_i\in L_i$ for $i=1, \dots, k$. Any Hamiltonian isotopy $(\varphi_t)_{t\in [0, 1]}$ between the identity and $\varphi\in \Ham_{\underline{L}}(\Sigma_{g, p})$ provides then a collection of $k$ curves $t\mapsto \varphi_t(P_i)$. Remark that there exists a permutation $\sigma$ in $\mathfrak{S}_{k+g}$ such that for all $i$, $\varphi(L_i)=L_{\sigma(i)}$, and contractible circles are mapped to contractible circles. For each contractible $L_i$, choose then a path in $L_{\sigma(i)}$ connecting $\varphi(P_i)$ to $P_{\sigma(i)}$. The concatenation of the curves $t\mapsto \varphi_t(P_i)$ with the respective connecting paths yields an element in $\mathcal{B}_{k, g, p}$. The braid obtained this way does not depend on the choice of the base points, of connecting paths or of the Hamiltonian isotopy (a proof of this fact is included later). We thus have defined a group homomorphism\begin{equation*}
b: \Ham_{\underline{L}}(\Sigma_{g, p})\rightarrow \mathcal{B}_{k, g, p}
\end{equation*}
Denote by $\mathcal{B}_{\underline{L}}$ the image of $b$, which will be described in detail in Section \ref{sec:image}. In particular, we show:
\begin{prp}
    $\mathcal{B}_{\underline L}$ is isomorphic to $\mathcal B_{k,0,p+2g}$.
\end{prp}
\begin{rem}
If $g\geq 1$, the braid type homomorphism $b$ cannot be surjective. Indeed, any non contractible circle in $\underline{L}$ is Hamiltonianly non-displaceable, and therefore has to be mapped to itself by any element in $\Ham_{\underline{L}}(\Sigma_{g, p})$. In the genus 0 case $b$ is surjective, see \cite{M}.
\end{rem}
Our main result is the following:

\begin{thm}\label{thm:braidPersistenceGenus}
There exists a family $(\mathfrak f_{(v_1, v_2)})_{v_i\in V}$ of group homomorphisms $\mathcal{B}_{\underline{L}}\rightarrow \R$ indexed on the square of a $p$-dimensional simplex $V$ (defined in Section \ref{sec:constrFunction}) such that for every $v_i\in V$ and $\varphi, \psi\in \Ham_{\underline{L}}(\Sigma_{g, p})$,
\begin{equation}
d_H(\varphi, \psi)\geq \frac 1 2 \abs{\mathfrak f_{(v_1, v_2)}(b(\varphi\psi^{-1}))}.
\end{equation}
In particular, we have the following estimate for the Hofer norm:
\begin{equation}
    \norm{\varphi}\geq \frac 1 2 \sup\limits_{(v_1, v_2)\in V^2}\abs{\mathfrak{f}_{(v_1, v_2)}(b(\varphi))}.
\end{equation}
Let us define the function on the image of $b$
\begin{equation}\label{eq:defnMathfrF}
    \mathfrak f(g)=\max_{(v_1, v_2)\in V^2}\abs{\mathfrak f_{(v_1, v_2)}(g)}.
\end{equation}
In the case of the generators $\sigma_j$, $a_i$, $b_i^{-1}a_ib_i$ we then obtain the explicit values:
\begin{equation*}
    \mathfrak f(\sigma_j)=\frac{1}{2}\mathfrak f(a_i)= \frac{1}{2}\mathfrak f(b_i^{-1}a_ib_i)=\frac{1}{2(k+g)}\frac{(k+1)\lambda-1}{k+2g-1}.
\end{equation*}
\end{thm}

\begin{rem}
    We compute in (\ref{eqn:f}) explicit values for the maximum in \ref{eq:defnMathfrF}
\end{rem}

\begin{rem}
This result seems counter-intuitive when compared to a result of Khanevsky's contained in \cite{kha15}. In that paper, Khanevsky defines the notion of homological trajectory of a Hamiltonian diffeomorphism defined on a surface and fixing a given disc. He then proceeds to prove that, whenever the genus of the surface is positive, there exists a constant only depending on the disc such that one can realise any trajectory with a Hamiltonian diffeomorphism of Hofer energy lower than this threshold. In fact, it is easily seen that such diffeomorphisms cannot stabilise all non-contractible component of the Lagrangian link $\underline{L}$, therefore they do not appear in our treatment. If anything, this shows that to find a result like ours, the constraint of fixing a number of non-contractible circles equal to the genus is in fact minimal (otherwise it is possible to construct those diffeomorphisms, realising arbitrarily complex braids with bounded Hofer energy).
\end{rem}
We proceed now to define a norm on $\mathcal{B}_{\underline L}$. For any $g\in \mathcal{B}_{\underline L}$ we set
\begin{equation}
    \norm{g}_{\underline{L}}:=\inf_{\varphi\in\Ham_{\underline{L}}(\Sigma_{g, p}), b(\varphi)=g}\norm{\varphi}
\end{equation}
It is easy to check that $\norm{\cdot}$ is a pseudonorm on $\mathcal{B}_{\underline L}$.

A corollary of the main result from this paper is an estimate from below of $\norm{\cdot}_{\underline L}$:

\begin{cor}\label{cor:hofNormBraids}
Let $g\in \mathcal{B}_{\underline{L}}$. Then, if $\mathfrak f$ is the function from the main result,
\begin{equation}
\norm{g}_{\underline L}\geq \frac{1}{2}\cdot\abs{\mathfrak{f}(g)}
\end{equation}
\end{cor}
As in \cite{M}, this is not enough to conclude non-degeneracy as soon as $k+g\geq 2$ (as soon as the genus is positive, there are no interesting cases we can consider for $k+g\leq 2$). We may however adapt a proof by Chen contained in \cite{chen23} to show that 
\begin{thm}
    For a pre-monotone $\underline L$, the pseudonorm $\norm{\cdot}_{\underline L}$ is non-degenerate.
\end{thm}
\noindent
Chen had in particular proved this fact for braids on the disc.

Recall the existence of the real-valued Calabi morphism on $\Ham_c(\Sigma_{g, p})$ (for a definition, see for instance \cite[Section 10.3]{mcDSal16}). An immediate corollary of this result is the existence of Hamiltonian diffeomorphisms in the kernel of Calabi with non-zero asymptotic Hofer norm.

For any $\varphi\in \Ham_c(\Sigma_{g, p})$, we define its asymptotic Hofer norm by
\begin{equation*}
    \norm{\varphi}_{\infty}:=\lim_n\frac{\norm{\varphi^n}}{n}
\end{equation*}
We may now state the following elementary consequence of Theorem \ref{thm:braidPersistenceGenus}:
\begin{cor}
 Let $\varphi\in \Ham_{\underline L}(\Sigma_{g, p})$. Then
 \begin{equation*}
     \norm{\varphi}_\infty\geq \frac{1}{2}\cdot\abs{\mathfrak f (b(\varphi))}
 \end{equation*}
Moreover, any braid type in the image of $b$ may be realised by a diffeomorphism in the kernel of Calabi.
\end{cor}
\begin{proof}
Since $\mathfrak f_{(v_1, v_2)}$ and $b$ are both morphisms of groups for all $v_i\in V$, the only statement which is not entirely trivial is the one about the kernel of Calabi. This is easily seen, since Calabi is a morphism of groups, and we may compose any element in $\Ham_{\underline L}(\Sigma_{g, p})$ by a diffeomorphism realising the trivial braid and arbitrary Calabi.
\end{proof}
\subsection{Structure of the paper}
We are going to start our paper by recalling the necessary notions from the construction of Quantitative Heegaard-Floer homology as in \cite{CGHMSS22}.
In Section \ref{sec:monotonicity} we are going to prove monotonicity of Lagrangian links under more general hypotheses than those in \cite{CGHMSS22}: as anticipated above, this is a crucial new property needed for the good functioning of the proof. In Section \ref{sec:kunneth} we recall some fundamental information about the Künneth isomorphism in \cite{MT}. It is possible to use such isomorphism to prove that the Floer Homology we work with is non-zero, and we shall also use it in a simplified proof of a particular case of our Theorem \ref{thm:braidPersistenceGenus} (this may be found at the end of the paper and is independent of the rest).  After reviewing this construction, in Sections \ref{sec:defB} and \ref{sec:image} we give some necessary details about braid and Hamiltonian diffeomorphism groups, details which are used in the construction of the Hofer-Lipschitz functions and in the ``braid type'' function\begin{equation*}
    b: \Ham_{\underline{L}}(\Sigma_{g, p})\rightarrow \mathcal{B}_{k, g, p}
\end{equation*}
It is in Section \ref{sec:image} in particular that we give a characterisation of the image of $b$ using the description in \cite{bel04}.

After this necessary introduction, in Section \ref{sec:proofMainResult} we go on to construct the family of the Hofer Lipschitz homomorphisms $\mathfrak f_{(v_1, v_2)}$, and explain the optimisation process which yields the function $\mathfrak{f}$ from the main result. In Section \ref{sec:intersections} we also describe in detail the intersection phenomena taking place in the symmetric product. In Section \ref{sec:chenAdapt} we show how the proofs contained in \cite{chen23} also apply in our setting, and prove non-degeneracy of the norms on the braid groups of surfaces we defined above. Section \ref{sec:proof_Kunneth} closes the paper with a simplified proof of our result in which we only consider Hamiltonian diffeomorphisms with compact support in a fixed disc: Section \ref{sec:proof_Kunneth} will only rely on the ideas from \cite{M} and the Künneth isomorphism in \cite{MT}. 

\paragraph{Acknowledgements}
The first author would like to thank Vincent Humilière for his support and interest in the project, Pierre Godfard for office chats about symmetric products of Riemann surfaces, Paolo Bellingeri and Cheuk Yu Mak for their helpful indications. Fabio Gironella's patient Inkscape tutorial also saved him a lot of time.
The second author would also like to thank Vincent Humilière, Cheuk Yu Mak as well as his advisor Sobhan Seyfaddini for his support and helpful suggestions.
We thank the referee for their patient work and suggestions that brought to a substantial improvement of the paper, and especially for proposing an alternative, simpler proof for Lemma \ref{lem:locality}. Both authors are supported by the ERC Starting Grant number 851701.

\section{Quantitative Heegaard-Floer homology and a Künneth Formula}
\label{sec:prelim}

We recall the construction of Heegaard-Floer homology and link spectral invariants. After giving an overview of the theory and providing some of its properties, in Section \ref{sec:monotonicity} we carry out the proof of monotonicity of the class of Lagrangian links we are going to use, and in Section \ref{sec:kunneth} we sketch the proof of a Kunneth formula as contained in \cite{MT}. We mainly follow \cite[Section 6]{CGHMSS22}, although we are considering homology instead of cohomology, and our conventions differ slightly from theirs. Moreover, we weaken the necessary conditions for the definition of the homology in the case of surfaces with genus.

Let $\Sigma_g$ be a closed, symplectic surface, and $\underline{L}$ be a collection of $k+g$ circles satisfying the conditions detailed in Definition \ref{defn:pre-monLink} setting $p=0$. Let $\varphi\in \Ham(\Sigma_g)$ be a Hamiltonian diffeomorphism defined by a Hamiltonian $H$, and such that $\forall i, j=1,\dots, k+g$, $\varphi(L_i)\pitchfork L_j$. 

The symmetric group on  $k+g$ letters acts on $(\Sigma_{g})^{k+g}$ by permutation of the coordinates. We denote indifferently by $\Sym^{k+g}(\Sigma_g)$ or $X$ the quotient of the cartesian product by this action: this set may be given the structure of a symplectic manifold, as in \cite{Perutz}. We consider then the image of the Lagrangian torus $\prod_{i=1}^{k+g}L_i$ under the quotient map: since $\prod_{i=1}^{k+g}L_i$ does not intersect the diagonal (the subset of the Cartesian product where two or more coordinates coincide), the quotient map induces a diffeomorphism between $\prod_{i=1}^{k+g}L_i$ and its image, which we denote by $\Sym^{k+g}(\underline{L})$. Let also $\Delta$ be the image of the diagonal under the quotient map in $\Sym^{k+g}\Sigma_g$.  The push-forward of the product symplectic form on $\Sigma_g^{k+g}$, which we denote by $\omega_X$, is not smooth near $\Delta$: in \cite{Perutz} the author applies a convolution to smoothen it. The smoothened form coincides with $\omega_X$ outside an arbitrarily small neighbourhood of the diagonal, so that $\Sym^{k+g}(\underline L)$ turns out to be Lagrangian. We shall still denote the smoothened form by $\omega_X$: this abuse of notation will not affect any of the following.

In a similar way, to $\varphi$ we may associate a Hamiltonian diffeomorphism of $\Sym^{k+g}\Sigma_g$: it is going to be defined by approximating the function
\begin{equation*}
\Sym H_t: \Sym^{k+g}\Sigma_g\rightarrow \R,\,\,\,\Sym H_t([x_1, \dots, x_{k+g}]):=\sum_{i=1}^{k+g}H_t(x_i)
\end{equation*}
via another one which is a time-dependent constant near $\Delta$, and coincides with $H_t$ away from it. By an abuse of notation we denote by $\Sym^{k+g}(\varphi)$ the diffeomorphism generated by an approximation of the quotient Hamiltonian (none of the following constructions will depend on the actual approximation we choose). The Quantitative Heegaard-Floer Homology of the pair $(\underline{L}, H)$ is defined to be the traditional Lagrangian Floer Homology in $\Sym^{k+g}\Sigma_g$ of the pair $(\Sym^{k+g}(\underline{L}), \Sym^{k+g}(H))$.

We may define the Floer complex
\begin{equation*}
CF(\underline{L}, H)
\end{equation*}
To define the complex we consider a point $x\in \Sym^{k+g}(\underline L)$ as a constant path. Generators of the Floer complex are Hamiltonian chords from $\Sym^{k+g}(\underline L)$ to $\Sym^{k+g}(\underline L)$, with capping. We assume that such are homotopic to the reference path $x$, and a capping is the choice of a homotopy between the Hamiltonian chord and $x$. The differential of the complex is defined by counting appropriate pseudo-holomorphic strips in $X$ with boundary conditions on the Lagrangian tori $\Sym^{k+g}(\underline L)$ and $\Sym^{k+g}(\varphi)\Sym^{k+g}(\underline L)$. For details, we refer to \cite{CGHMSS22} (our conventions here are consistent with those of \cite{MT}).

The complex is filtered by the symplectic action: given a Hamiltonian chord $y$ with capping $\hat{y}$, its action is
\begin{equation}\label{eq:action}
\mathcal{A}^\eta_H(\hat{y}):=\int_0^1 \Sym H_t(y(t))\, dt - \int_{[0, 1]\times[0, 1]}\hat{y}^*\omega_X - \eta[\hat{y}]\cdot \Delta
\end{equation}
The constant $\eta$ appearing in the formula for the action parametrises a deformation one has to apply to the symplectic form on the symmetric product obtained in \cite{Perutz} in order to make the quotient Lagrangian monotone: for a detailed explanation see \cite[Remark 4.22]{CGHMSS22}. By \cite[Lemmata 3.3, 3.4]{M} (which still hold in the context of this paper) the unfiltered complex shall not depend on $\eta$, hence we suppress it from the notation. The action filtration as expected however does, and showing that its dependence on $\eta$ is connected to the braid type of a Hamiltonian diffeomorphism (if such braid type is defined) is in fact one key argument towards our proof. The actual value of $\eta$ is easily recovered from the area of the non-contractible component $B_{k+1}$:
\begin{equation*}
    \mathrm{Area}(B_{k+1})+2\eta(k+2g-1)=\lambda.
\end{equation*}
An explanation for this relation in the case of a surface without boundary (which is the case we are interested in) is the object of Section \ref{sec:monotonicity}.
The very last term is the intersection product between the capping and the diagonal: this is well defined because the capping sends the boundary of the square to the Lagrangians, which do not intersect the diagonal by assumption. The group $\pi_2(X)$ acts on the space of capped Hamiltonian chords, given by recapping: since the Lagrangian $\Sym^{k+g}(\underline L)$ is monotone, the symplectic action is shifted by multiples of $\lambda$ by recapping.

\begin{prp}
    The homology $HF(\underline L, H)$ is well defined.
\end{prp}
\begin{proof}
    We apply Proposition \ref{prp:monotonicity} (proved in Section \ref{sec:monotonicity}) and the arguments contained in \cite[Section 6]{CGHMSS22}.
\end{proof}
\begin{prp}
    The homology of $CF(\underline L, H)$ is non-zero.
\end{prp}
\begin{proof}
    The proof is a simple application of the Künneth formula together with continuation maps for the homology, see Section \ref{sec:kunneth} for the explanation of the argument.
\end{proof}
\begin{rem}
    We highlight that in the case of links with $g$ non-contractible components we may not use the tools from \cite{CGHMSS22} to deduce that the differential squares to 0, and that even in such a case the homology is non-zero. In fact, they make the stronger assumption on the Lagrangian link that the closures of the connected components $B_i$ be planar. If we have exactly $g$ non-contractible components however, this hypothesis is not satisfied, and it is not clear in principle whether the Floer homology in the setting we analyse in this paper is defined, or non-zero. In Sections \ref{sec:monotonicity} and \ref{sec:kunneth} both questions are addressed.
\end{rem}

The homology admits a product structure with unit $\mathbf{1}$. The product operation is, as per usual in Lagrangian Floer Theory, defined via counts of pseudo-holomorphic triangles: this yields a map
\begin{equation*}
    CF(\underline{L}, H)\otimes CF(\underline{L}, H)\rightarrow CF(\underline L, H\# H).
\end{equation*}One then post-composes the arrow above by a continuation map
\begin{equation*}
    CF(\underline L, H\# H)\rightarrow CF(\underline L, H)
\end{equation*}
and the composition descends to the homological level. The resulting operation does not moreover depends on the choices one makes to define the complexes and the continuation map.

We define the spectral invariant $c_{\underline L}(H)$ as:
\begin{equation*}
    c_{\underline L}(H):=\frac{1}{k+g}\inf\Set{\alpha\in \R \vert \mathbf{1} \text{ may be represented by a chain of action} <\alpha}
\end{equation*}
If the Hamiltonian $H$ is mean-normalised in fact $c_{\underline L}(H)$ is an invariant of the homotopy class of the Hamiltonian path $t\mapsto \varphi_t$ with fixed endpoints. If $g\geq 1$, the group of Hamiltonian diffeomorphisms is simply connected, so that $c_{\underline L}$ provides an invariant of the time 1-map. The resulting map\begin{equation*}
    c_{\underline L}: \Ham(\Sigma_g)\rightarrow \R
\end{equation*}
is moreover Hofer Lipschitz. The reason why $c_{\underline L}$ is defined on the whole of $\Ham(\Sigma_g)$ and not just on the non-degenerate diffeomorphisms is given by Hofer-density of the latter set, and by the Hofer Lipshitz property.

\subsection{Proof of monotonicity}\label{sec:monotonicity}
To show that the differential on the complex is well defined, we need to show that the zero-dimensional moduli spaces are compact. This is done in \cite{CGHMSS22} by proving a monotonicity property which gives a uniform bound to the energy of Maslov index 1 strips with given asymptotics. However, they only prove the property for links such that the closure of each connected component of the complement is planar. In this section, we prove the monotonicity property for a larger class of links, which in turn imply that the Floer complex is well defined.

For more generality, we assume that $\int_{\Sigma_g}\omega=A$ (the area $A$ is not necessarily 1). In the following, we split the circles in $\underline{L}$ in two classes: the contractible and the non-contractible ones. We assume that  $L_1,...,L_k$ ($k\geq 2$) bound disjoint discs $B_1,...,B_k$ of area $\lambda\in [\frac A {k+1}, \frac A k)$, and that $\alpha_1,...,\alpha_g$ ($\alpha_i:=L_{k+i}$) be meridians for each handle of $\Sigma_g$. This way, we have $\underline L=L_1\cup...\cup L_k\cup\alpha_1\cup...\cup\alpha_g$ and $X:= \Sym^{k+g}(\Sigma_g)$.
Let $B_{k+1}$ be the only connected component of $\Sigma_g\setminus\underline L$ that is not a disc, and let $A_{k+1}$ be its area. Let $\eta$ be the real number satisfying
\[A_{k+1}+2\eta(k+2g-1)=\lambda\]
Then $\eta$ can be recovered by the formula \begin{equation*}
   \eta=\frac {\lambda-A_{k+1}}{2(k+2g-1)}=\frac{(k+1)\lambda-A}{2(k+2g-1)} 
\end{equation*}and in particular is non-negative.

\begin{prp}
    \label{prp:monotonicity}
    For all $[u]$ in the image of $\pi_2(X,\Sym^{k+g}(\underline L))\to H_2(X,\Sym^{k+g}(\underline L))$ one has $$\omega_X([u])+\eta\Delta\cdot [u] = \frac \lambda 2 \mu([u])$$
    where $\mu$ is the Maslov class $H_2(X,\Sym^{k+g}(\underline L))\rightarrow\Z$.
\end{prp}

\begin{proof}
    For $1\leq i \leq k+1$, let $\overline B_i$ be the closure of $B_i$ in $\Sigma_g$. Let $k_i$ be the number of boundary components of $\overline B_i$. For each $i$, fix a point $a_i$ in $B_i$, and let $X_{a_i}$ be the projection of $\Sigma_g^{k+g-1}\times \{a_i\}$ in $\Sym^{k+g}(\Sigma_g)$.
    In \cite[Section 4.5]{CGHMSS22}, it is explained how, when $\overline B_i$ is planar, one can construct a disc class $[u_i]$ in $H_2(X,\Sym\underline L)$, which satisfies $[u_i]\cdot X_{a_j}=\delta_{i,j}$, $[u_i]\cdot \Delta = k_i-1$, and $\mu([u_i])=2$.
    Here, the $\overline B_i$'s for $1\leq i \leq k$ are discs, therefore we can apply this construction to get $k$ classes satisfying $\omega_X([u_i])+\eta\Delta\cdot [u_i] = \lambda + 0 = \frac \lambda 2 \mu([u_i])$.

    However, $\overline B_{k+1}$ is not a planar domain (it is equal to the whole surface minus $k$ disjoint discs). But we can still apply a similar construction:

    Let $\widehat D:=\overline B_{k+1}\sqcup\coprod_{1\leq j\leq g}D_j$, where the $D_j$'s are copies of the closed unit disc D in $\C$.
    Let $\pi_{\widehat D}:\widehat D\to D$ be a $(k+g)$-fold simple branched covering, such that $\pi_{\widehat D}|_{D_j}$ is a biholomorphism for all $j$, and $\pi_{\widehat D_i}|_{\overline B_{k+1}}$ is a topological $k$-fold simple branched cover, such that $a_{k+1}$ is not a branched point. Since $k\geq 2$, such a branched cover always exists; one can construct it using techniques presented in John Etnyre's lecture notes \cite{Etnyre}.
    Let $v_{k+1}:\widehat D\to \Sigma_g$ be a map whose restriction to $\overline B_{k+1}$ is the identity, and which sends $D_j$ to a point in $\alpha_j$.
    Then, by tautological correspondence, we get a map $u_{k+1}:(D,\partial D)\to (X,\Sym\underline L)$, defined by $u_{k+1}(z)=v_{k+1}(\pi_{\widehat D}^{-1}(z))$. The class $[u_{k+1}]$ also satisfies $[u_{k+1}]\cdot X_{a_i}=v_{k+1}\cdot a_i=\delta_{i,k+1}$ for $1\leq i \leq k+1$.
    Therefore, the proof of \cite[Lemma 4.10]{CGHMSS22} goes through, and we get that the image of $\pi_2(X,\Sym\underline L)\to H_2(X,\Sym \underline L)$ is freely generated by $\{[u_i]\}_{i=1}^{k+1}$, and that the image of $\pi_2(X)\to H_2(X,\Sym \underline L)$ is freely generated by $\sum_{i=1}^{k+1}[u_i]$.

    Since the $[u_i]$ satisfy the equation $\omega_X([u_i])+\eta\Delta\cdot [u_i] = \frac \lambda 2 \mu([u_i])$ for $1\leq i \leq k$, it only remains to show it for $[u_{k+1}]$.

    We start by computing $[u_{k+1}]\cdot\Delta$ in a similar fashion as in \cite{CGHMSS22}. Let $\pi:\Sigma_g\to S^2$ be a topological $k+g$-fold simple branched covering of the sphere. Let $u:S^2\to X$ be the map tautologically corresponding to the pair $(\pi,id_{\Sigma_g})$. Then, since $[u]$ is in the image of $\pi_2(X)\to H_2(X,\Sym \underline L)$, there exists an integer $c$ such that $[u]=c\sum_{i=1}^{k+1}[u_i]$. Since $[u]\cdot X_{a_i} = 1$ for all $i$, we get that $c=1$. Therefore, $[u]\cdot \Delta = (\sum_{i=1}^{k+1}[u_i]) \cdot\Delta=[u_{k+1}]\cdot \Delta$.
    Since $id_{\Sigma_g}$ is injective, there is a one-to-one correspondence between branched point of $\pi$ and points of $S^2$ whose image by $u$ lies in $\Delta$. Moreover, since the branched points are simple, we have that $[u]\cdot \Delta$ is actually equal to the number of branched points of $\pi$.
    Therefore, the Riemann-Hurwitz formula implies that $2-2g=2(k+g)-[u]\cdot\Delta$, and we finally get $[u_{k+1}]\cdot\Delta = 2(k+2g-1)$.
    It only remains to compute the Maslov index of $[u_{k+1}]$:
    \[\mu([u_{k+1}])= \mu\left([u]-\sum\limits_{i=1}^{k}[u_i]\right)=2\langle c_1(TX),[u]\rangle-2k\]
    
    Let $v$ be the Abel-Jacobi map, from $\Sym^d(\Sigma_g)$ to its Jacobian variety
    $J$, which is isomorphic to the $2g$-dimensional torus $T^{2g}$.
    According to \cite[Chapter VII, Section 5]{Arbarello}, there is a class $\theta$ in $H^2(J)$ and a point $q$ in $\Sigma_g$ such that:
    \[c_1(T\Sym^d\Sigma_g)=(d-g+1)PD(X_q)-v^*\theta\]

    Now, in our case, $d=k+g$, and as above $[u]$ is a generator of the image of $\pi_2(X)\to H_2(X,\Sym \underline L)$. Since $J=T^{2g}$ is aspherical, $v_*[u]$ vanishes, and
    \[\mu([u_{k+1}])=2\langle c_1(TX),[u]\rangle -2k=2(k+1)[u]\cdot X_q-2k = 2\] where the last equality comes from the fact that $[u]\cdot X_q = id_{\Sigma_g}\cdot q=1$ when $q$ is not a branched point of $\pi$, which we can assume by perturbing $\pi$ if necessary.
    As claimed, we get
    \[\omega_X([u_{k+1}])+\eta\Delta\cdot [u_{k+1}] = A_{k+1}+2\eta(k+2g-1)=\lambda = \frac \lambda 2 \mu([u_i])\]    
\end{proof}

\begin{rem}
    One can use the results of this section to show the monotonicity property for a more general class of links. In fact, let $\underline L$ be a link on $\Sigma_g$, and $B_1,...,B_s$ be the connected components of $\Sigma_g\setminus\underline L$. Let $k_i$ be the number of boundary components of $\overline B_i$ (the closure of $B_i$ in $\Sigma_g$), $g_i$ be its genus, and $A_i$ be its area. Assume the following is satisfied:
    \begin{itemize}
        \item for all $i$, $\overline B_i$ contains exactly one non-contractible component of $\underline L$ around each of its handles (otherwise said, the non-contractible components of the link generate a $g$-dimensional subspace of the first homology of $\Sigma_g$);
        \item if $g_i>0$, $k_i\geq 2$;
        \item there exists constants $\lambda$ and $\eta$ such that for all $1\leq i \leq s$, $A_i+2\eta(k_i+2g_i-1)=\lambda$.
    \end{itemize}
    Then, $\underline L$ satisfies the monotonicity property: for all $[u]$ in the image of \begin{equation*}
        \pi_2(X,\Sym(\underline L))\to H_2(X,\Sym(\underline L))
    \end{equation*}one has $\omega_X([u])+\eta\Delta\cdot [u] = \frac \lambda 2 \mu([u])$.
    As in \cite{CGHMSS22}, it is enough to check the monotonicity property on the basic disc classes tautologically corresponding to the $\overline B_i$'s. When they are planar ($g_i=0$), it is shown in \cite{CGHMSS22} that they satisfy the monotonicity property by embedding $\overline B_i$ into the sphere. If $g_i\neq 0$, we can embed $\overline B_i$ into $\Sigma_{g_i}$ by gluing $k_i$ discs of area $\lambda$ on the boundary components. Then, we can apply the result of this section to get the monotonicity property for $\overline B_i$.
\end{rem}

\subsection{The K\"unneth formula}\label{sec:kunneth}
    In this section we present a result shown by Cheuk Yu Mak and the second author in \cite{MT}. We will use it in Section \ref{sec:proof_Kunneth} to give an easier proof of our main result in the (more restrictive) case of disc supported Hamiltonians. It was also used above to prove that the homology theory we consider is not 0. The setting is the following: let $\underline L$ and $\underline K$ be two transverse $\eta$-monotone Lagrangian links on a closed surface $(\Sigma,\omega)$. Let $(E,\omega_E)$ be a two-dimensional torus, $\alpha$ be a non-contractible circle on $E$, and $\alpha'$ be a small Hamiltonian deformation of $\alpha$, transverse to $\alpha$. Pick a point $\sigma_1\in\Sigma$ away from $\underline L$ and $\underline K$, and a point $\sigma_2\in E$ away from the support of the isotopy between $\alpha$ and $\alpha'$. Pick open discs $B(\sigma_1)$ and $B(\sigma_2)$ that contain $\sigma_1$ and $\sigma_2$ respectively, and that are also away from the Lagrangian links and the isotopy. Let $\Sigma':=\Sigma\#E$ be the connected sum of $\Sigma$ and $E$, obtained by gluing together $\Sigma\setminus \{\sigma_1\}$ and $E\setminus\{\sigma_2\}$ along the two cylinders $B(\sigma_1)\setminus\{\sigma_1\}$ and $B(\sigma_2)\setminus\{\sigma_2\}$. We equip $\Sigma'$ with an area form $\omega'$ which coincides with $\omega$ on $\Sigma\setminus B(\sigma_1)$, with $\omega_E$ on the support of the isotopy between $\alpha$ and $\alpha'$, and such that $\omega'(\Sigma')+4\eta=\omega(\Sigma)$. (For this last condition to be satisfied, we need to assume that the support of the isotopy between $\alpha$ and $\alpha'$ has small enough area.)
    The result is the following (\cite[Theorem 6]{MT}): for an appropriate choice of complex structure, there is a filtered chain complex isomorphism
    \[CF(\underline L,\underline K)\otimes CF(\alpha,\alpha')\xrightarrow{\sim}CF(\underline L\cup\alpha,\underline K\cup\alpha')\]
    In this K\"unneth formula for the connected sum, $\underline L$ and $\underline K$ are viewed as links on $\Sigma$, $\alpha$ and $\alpha'$ are Lagrangians on $E$, and $\underline L\cup\alpha$ and $\underline K\cup\alpha'$ are links in $\Sigma'$. The action for $CF(\alpha,\alpha')$ is also defined by the formula \ref{eq:action}, but in this case $\Delta=\emptyset$.
    The proof of this statement is inspired by the ``stabilisation invariance'' shown in \cite[Section 10]{OS}. The first step is to identify the generators of those vector spaces. Then to show that we get a chain isomorphism, one needs to identify the moduli spaces counting curves contributing to the differentials on both sides. This is done by a gluing argument and Gromov compactness. Finally, the part of the statement concerning the action filtration is a straightforward consequence of our choice of symplectic form $\omega'$.

    \begin{rem}Compared to \cite{MT} who showed the isomorphism in the case $\eta = 0$, we needed here to add the condition $\omega'(\Sigma')+4\eta=\omega(\Sigma)$ to make sure that the action filtration is preserved. This equality implies that $\underline L\cup\alpha$ and $\underline K\cup\alpha'$ are also $\eta$-monotone links on $(\Sigma',\omega')$, for the same $\eta$ as $\underline L$ and $\underline K$.
    \end{rem}

    We can use this K\"unneth isomorphism to show that the Heegaard Floer homology of a monotone link $\underline L$ (as in Definition \ref{defn:pre-monLink}) on a closed surface $\Sigma_g$ does not vanish. Indeed, one can write $\underline L = \underline L'\cup\alpha_1\cup...\alpha_g$ and decompose $\Sigma_g$ as a connected sum of a sphere containing $\underline L'$ and $g$ tori each containing a circle $\alpha_i$. $\underline L'$ is then a monotone link on the sphere, and by \cite{CGHMSS22} its homology $HF(\underline L',\underline L')$ does not vanish. Applying the K\"unneth isomorphism $g$ times, we get that $HF(\underline L,\underline L)$ does not vanish either. Finally, by invariance of the homology, it also does not vanish for any choice of Hamiltonian $H$.
    Alternatively, one could try to directly prove the non-vanishing of the homology as in \cite{CGHMSS22} by computing the potential of $\Sym(\underline L)$ and showing it admits a non-degenerate critical point. This potential is a map $W:H^1(\Sym(\underline L),\C^*)\to \C$ whose definition involves counting holomorphic discs of Maslov index 2 with boundary on $\Sym \underline L$ (see \cite{FOOO} for details about the theory of the potential).

\section{The braid type homomorphism}
\subsection{The definition}\label{sec:defB}
In the introduction it was claimed that $b$ was a well defined homomorphism from $\Ham_c(\Sigma_{g, p})$. To prove it, we need to remark that $b$ does not depend on the Hamiltonian isotopy one chooses. This is true if, for instance, $\pi_1(\Ham_c(\Sigma_{g, p}))=0$. In fact, since in such a case any two Hamiltonian isotopies are homotopic relative endpoints between isotopies, we may deform one Hamiltonian path into the other keeping strands from crossing. Such a deformation then provides a braid isotopy between the images of the two Hamiltonian isotopies, and $b$ is well defined. We then just have to show that $\pi_1(\Ham_c(\Sigma_{g, p}))=0$: this fact is standard, but as we could not find a reference, we decided to include the following proof.
\begin{lem}\label{lem:HamSimplConn}
    If $\Sigma_{g, p}\neq \sphere^2$, $\Ham_c(\Sigma_{g, p})$ is simply connected.
\end{lem}
\begin{proof}
    Let $B=\partial \Sigma_{g, p}$: it is a disjoint union of $p$ circles. The embedding $B\hookrightarrow \Sigma_{g, p}$ provides a fibration
    \begin{equation}\label{eq:fibrDiff}
        \mathrm{Diff}_0(\Sigma_{g, p}, B)\hookrightarrow \mathrm{Diff}_0(\Sigma_{g, p})\rightarrow \mathrm{Diff}_0(B)
    \end{equation}
where the fibre is the group of diffeomorphisms of $\Sigma_{g, p}$ inducing the identity on the boundary and isotopic to the identity, the total space is the connected component of the identity of the diffeomorphisms of $\Sigma_{g, p}$, and the base is the connected component of the identity in the group of diffeomorphisms of a disjoint union of circles.

Let us consider the long exact sequence of the fibration in (\ref{eq:fibrDiff}):
\begin{equation}\label{eq:LESfibr}
    \pi_2(\mathrm{Diff}_0(B))\rightarrow\pi_1(\mathrm{Diff}_0(\Sigma_{g, p}, B))\rightarrow \pi_1(\mathrm{Diff}_0(\Sigma_{g, p}))\rightarrow \pi_1(\mathrm{Diff}_0(B))
\end{equation}
The term $\pi_2(\mathrm{Diff}_0(B))$ is always 0, since it is a power of $\pi_2(\mathrm{Diff}_0(\sphere^1))$ and this is trivial ($\mathrm{Diff}_0(\sphere^1)$ has the homotopy type of the circle itself). This proves that the second arrow in the exact sequence (\ref{eq:LESfibr}) is always an injection. Now, if $g>1$, $(g=1,p\geq 1)$ or $(g=0, p\geq 3)$, \cite[Theorem 1]{Gramain} shows that $\mathrm{Diff}_0(\Sigma_{g, p})$ is contractible: in such a case we deduce that $\pi_1(\mathrm{Diff}_0(\Sigma_{g, p}, B))=0$. We want to prove that in this case $\mathrm{Diff}_0(\Sigma_{g, p}, B)$ and $\Ham_c(\Sigma_{g, p})$ have isomorphic fundamental groups.

Let $\Symp_0(\Sigma_{g, p}, B)$ be the group of symplectic diffeomorphisms of $\Sigma_{g, p}$ inducing the identity on the boundary and isotopic to the identity through symplectic diffeomorphisms with the same property. By \cite{ban74} (generalising a result of Moser \cite{mos65}) $\Symp_0(\Sigma_{g, p}, B)$ is a deformation retract of $\mathrm{Diff}_0(\Sigma_{g, p}, B)$, and in particular\begin{equation*}
    \pi_1(\Symp_0(\Sigma_{g, p}, B))\cong\pi_1(\mathrm{Diff}_0(\Sigma_{g, p}, B))
\end{equation*}
Fix a decreasing sequence of open collar neighbourhoods of the boundary, call them $(U_n)_{n\geq 1}$: they satisfy the property
\begin{equation}\label{eq:intCollNeigh}
    \bigcap_{n\geq 1}U_n=B
\end{equation}
Define by $\Symp_{U_n}(\Sigma_{g, p})$ the group of symplectic diffeomorphisms of $\Sigma_{g, p}$ which are supported in $\Sigma_{g, p}\setminus U_n$ and are isotopic to the identity via such diffeomorphisms. 

Clearly, 
\begin{equation}
   \Symp_{U_n}(\Sigma_{g, p})\subset \Symp_{U_{n+1}}(\Sigma_{g, p}),\,\,\, \bigcup_{n\geq 1}\Symp_{U_n}(\Sigma_{g, p})=\Symp_c(\Sigma_{g, p})
\end{equation}
and both these conditions together imply that
\begin{equation}\label{eq:limFundGroupsSymp}
    \pi_1(\Symp_c(\Sigma_{g, p}))=\lim_n\pi_1(\Symp_{U_n}(\Sigma_{g, p}))
\end{equation}
We shall now prove that $\pi_1(\Symp_c(\Sigma_{g, p}))$ injects into $\pi_1(\Symp_0(\Sigma_{g, p}, B))=0$ via the map induced by the inclusion. From this it is possible to deduce that $\pi_1(\Ham_c(\Sigma_{g, p}))=0$, since the map:
\begin{equation*}
   \pi_1(\Ham_c(\Sigma_{g, p}))\rightarrow  \pi_1(\Symp_c(\Sigma_{g, p}))
\end{equation*}
defined by the inclusion is in fact an injection (see \cite[Proposition 10.2.13]{mcDSal16}).

Let therefore $t\mapsto \varphi_t\in \Symp_c(\Sigma_{g, p})$ be a loop based at the identity which becomes contractible once seen as a representative in $\pi_1(\Symp_0(\Sigma_{g, p}, B))$. There is therefore a homotopy of symplectic isotopies fixing the boundary pointwise shrinking $t\mapsto \varphi_t$ to the constant path at the identity:
\begin{align*}
    \psi: [0, 1]_s\times [0, 1]_t\rightarrow\Symp_0(\Sigma_{g, p}, B),\\\,\,\, \psi(0, t)=\varphi_t, \psi(1, t)=\mathrm{Id}, \psi(s, 0)=\psi(s, 1)=\mathrm{Id}
\end{align*}

The goal now is to modify this isotopy to a new one, in area-preserving diffeomorphisms which fix a small enough collar neighbourhood of the boundary. Since the support of $\varphi_t$ is compact in $\Sigma_{g, p}\setminus B$ for all $t$, there exists an $n$ big enough such that the support of $\varphi_t$ is contained in $\Sigma_{g, p}\setminus U_n$ for all $t$.

Since we know that $\Symp_0(\Sigma_{g, p}, B)$ is a deformation retract of $\Diff_0(\Sigma_{g, p}, B)$, we may homotope $\psi$ to another homotopy $\psi'$, this time via diffeomorphisms fixing $U_n$ pointwise (but not necessarily area-preserving), from the identity to itself. We now consider $\psi'$ as a homotopy in diffeomorphisms of $\overline{\Sigma_{g, p}\setminus U_n}$: applying Moser's result again we find a homotopy in $\Symp_{U_n}(\Sigma_{g, p})$, call it $\psi''$, between the symplectic loop $t\mapsto \varphi_t$ and the constant loop. Since $\Symp_{U_n}(\Sigma_{g, p})\subset \Symp_c(\Sigma_{g, p})$, the existence of $\psi''$ proves that we have a sequence of group morphisms
\begin{equation*}
     \pi_1(\Ham_c(\Sigma_{g, p}))\hookrightarrow  \pi_1(\Symp_c(\Sigma_{g, p}))\hookrightarrow\pi_1(\Symp_0(\Sigma_{g, p}, B))\cong\pi_1(\mathrm{Diff}_0(\Sigma_{g, p}, B))
\end{equation*}
and the rightmost group is 0 whenever $g>1$, $(g=1, p\geq 1)$ or $(g=0, p\geq 3)$ by Gramain's result. We have thus proved that $\Ham_c(\Sigma_{g, p})$ is simply connected under the above topological assumptions.

If $g=1, p=0$, $\Sigma_{g, p}$ is a torus, for which it is well known that $\Ham_c(\Sigma_{1, 0})$ is simply connected (see for instance \cite[Section 7.2]{pol01}).

We are now left with the cases of the disc $\D$ and cylinder $Z$ to consider. In these two cases \cite{Gramain} shows that\begin{equation*}
    \Diff_0(\D, \partial \D)\sim \Diff_0(Z, \partial Z)\sim SO(2, \R)
\end{equation*}
where $\sim$ denotes homotopy equivalence, the diffeomorphisms induce the identity on the boundary, and $SO(2, \R)$ is the real special orthogonal group of rank 2. Let us again adopt the notation from above: $\Sigma_{g, p}$ is our surface (disc or cylinder), and $B$ its boundary. We have again an exact sequence \begin{equation*}
    0\rightarrow\pi_1(\mathrm{Diff}_0(\Sigma_{g, p}, B))\rightarrow \pi_1(\mathrm{Diff}_0(\Sigma_{g, p}))\rightarrow \pi_1(\mathrm{Diff}_0(B))
\end{equation*}
but here the third group is not trivial. We need thus to show that the image of the second arrow is 0.

If $\Sigma_{g, p}=\D$, then $SO_2(\R)$ is a subgroup of $\mathrm{Diff}_0(\Sigma_{g, p})$, and is in fact its deformation retract. A generator of $\pi_1(SO_2(\R))$ is given by a full rotation, which is mapped to the generator of $\pi_1(\Diff_0(\partial \D))$: the third arrow is an injection in the case of the disc, so that the map
\begin{equation*}
    \pi_1(\mathrm{Diff}_0(\Sigma_{g, p}, B))\rightarrow \pi_1(\mathrm{Diff}_0(\Sigma_{g, p}))
\end{equation*}
from above is 0 as claimed.

If we examine $\Sigma_{0, 2}=Z$ instead, the argument is similar: the generator of $\pi_1(\Diff_0(Z))$ is the full rotation, which is mapped to a non-zero element in $\pi_1(\Diff_0(\partial Z))$ (not a generator in this case, of course). The third arrow is an injection in this case as well.

Summing up, in the cases $g=0, p\in \{1, 2\}$ we still have $\pi_1(\Diff_0(\Sigma_{g, p}, B))=0$. The rest of the arguments above did not depend on the actual surface we worked on, and they carry over to this context: we infer that\begin{equation*}
  \pi_1(\Ham_c(\D))\cong\pi_1(\Ham_c(Z))=0  
\end{equation*}
\end{proof}
\begin{rem}
    In fact, one should be able to deduce from the proof that $\Ham_c(\Sigma_{g, p})$ is contractible whenever $\Diff_0(\Sigma_{g, p}, B)$ is. One may apply \cite[Theorem 15]{pal66} and use it in combination with the fact that the homotopy groups of an increasing union of topological spaces is the limit of homotopy groups of said subspaces (easy consequence of the compactness of spheres).
\end{rem}

\subsection{The image}
\label{sec:image}
Let $\Sigma_{g, p}$ be a Riemann surface with genus $g$ and $p$ punctures. Bellingeri in \cite{bel04} describes via generators and relations the braid group with $k$ strands on $\Sigma_g$, for which we write $\mathcal{B}_{k, g, p}$.

Such group has $k+2g+p-2$ generators, denoted by 
\begin{equation}
    \sigma_1,\dots, \sigma_{k-1}, a_1,\dots, a_g,b_1,\dots, b_g, z_1,\dots, z_{p-1}
\end{equation}
The $\sigma_i$ here represent a clockwise half-twist between $P_i$ and $P_{i+1}$ (while the other strands are fixed), the $a_i$ move $P_1$ along a non-contractible loop on a generator of the homology parallel to $\beta_i$, the $b_i$ move $P_1$ along a loop parallel to $\alpha_i$, and $z_i$ move $P_1$ along a non-contractible loop going around the $i-th$ puncture. The $a_i$, $b_i$ and $z_i$ all keep the points $P_2,\ldots,P_k$ fixed. These generators are depicted in Figure \ref{fig:generators}.

\begin{figure}
\centering
\def\svgwidth{\linewidth}
     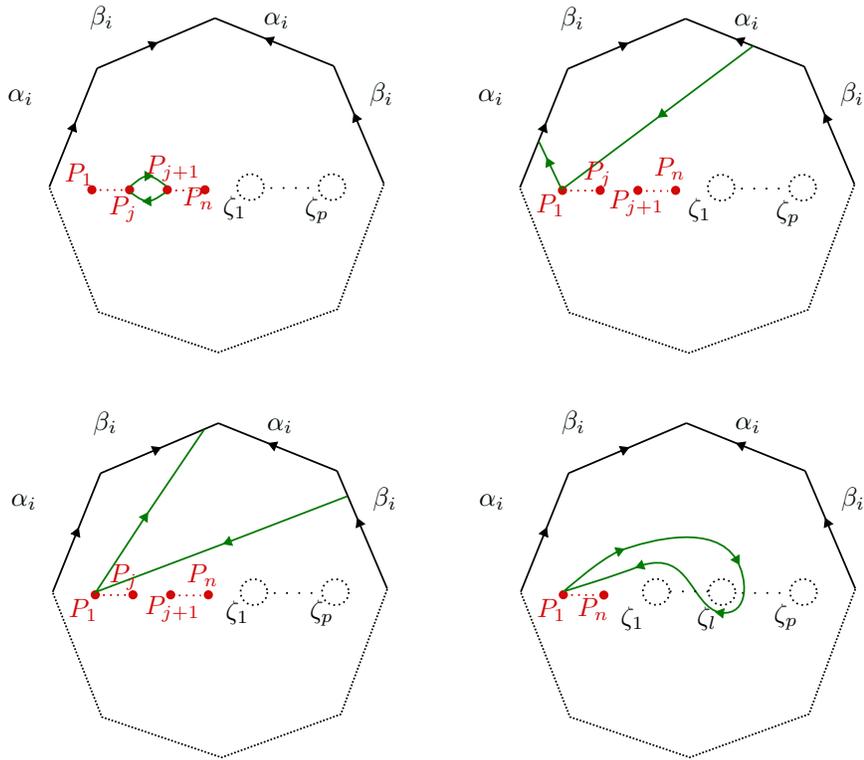
     \caption{The generators of the braid group $\mathcal{B}_{k, g, p}$. In clockwise order, starting from top-left: $\sigma_j$, $a_i$, $z_l$ and $b_i$ . We draw a fundamental domain of the surface, and all base points which are not endpoints of the green paths are to be thought of as constant paths.}
        \label{fig:generators}
\end{figure}

There are then 8 groups of relations, we refer to \cite[Theorem 1.1]{bel04} for the precise definitions.

We aim to describe the braid type of a Hamiltonian diffeomorphism preserving a given premonotone Lagrangian configuration $\underline L$. We assume hereafter that the $P_i$'s lie in $L_i$ and that the $g$ non contractible circles in $\underline L$ coincide with small disjoint deformations of the $\beta_i$ (see for instance $L_{k+i}$ in Figure \ref{fig:b-1ab_generator}). Consider the subgroup of $\mathcal{B}_{k, g, p}$ generated by ${(\sigma_i)}_{i=1,\dots, k-1}$, ${(a_j)}_{j=1,\dots, g}$,  ${(b_j^{-1}a_jb_j)}_{j=1,\dots, g}$ and ${(z_l)}_{l=1,\dots, p-1}$ only (no $b_j$ alone is present), with the restrictions of Bellingeri's relations. For clarity, the $\sigma_i$ we consider are those describing exchanges between contractible components of the link only. Denote this subgroup by $\mathcal{B}_{\underline{L}}$: in Lemma \ref{lem:imb} we prove that, consistently with the notation given in the introduction, $\mathcal{B}_{\underline{L}}$ is indeed the image of the braid type homomorphism.

\begin{rem}
\label{rem:relation}
    Even though we have $p$ punctures, the associated generators are only $p-1$. As remarked in \cite{bel04}, one may express a single loop $z_p$ around the last puncture using the relation
    \[[a_1,b_1^{-1}]\cdots [a_g,b_g^{-1}]=\sigma_1\sigma_2\cdots\sigma_{k-1}^2\cdots \sigma_2\sigma_1z_1^{-1}\cdots z_p^{-1}\]
\end{rem}
\begin{lem}\label{lem:imb}
    Let us consider the group homomorphism
    \begin{equation}
        b:\Ham_{\underline{L}, c}(\Sigma_{g,p})\rightarrow\mathcal{B}_{k, g, p}
    \end{equation}
    Its image is precisely $\mathcal{B}_{\underline{L}}$, and it is isomorphic to $\mathcal B_{k,0,p+2g}$.
\end{lem}
\begin{proof}
We are going to describe the group $\mathcal{B}_{\underline{L}}$ as the fundamental group of the $k$-th configuration space of the surface $\tilde{\Sigma}_{g, p}$, which we obtain by removing tubular neighbourhoods of the $g$ non contractible circles $(L_{j})_{i=k+1, \dots, k+g}$ from $\Sigma_{g, p}$. Now, $\tilde{\Sigma}_{g, p}$ is a punctured sphere, with $2g+p$ punctures. Its braid group $\mathcal B_{k,0,p+2g}$ is described in \cite[Theorem 2.1]{beFu04}: it has the $k-1$ generators corresponding to moves which take place in a disc on the surface, and $p+2g-1$ corresponding to non-contractible loops around punctures.

Let us show that we may identify the image of $b$ with this fundamental group. Let $\varphi\in\Ham_{\underline{L}, c}(\Sigma_{g, p})$, $(\varphi_t)_{t\in[0,1]}$ any Hamiltonian isotopy between the identity and $\varphi$, and let $L_j$ be any non contractible component of $\underline L$. In Lemma \ref{lem:loopHamDif}we are going to show that there is a loop in $\Ham_c(\Sigma_{g, p})$ based at the identity, $t\mapsto \psi_t$, such that for all $t$ and for any non contractible component $L_j$ one has
\begin{equation*}
    \psi_t\varphi_t(L_j)=L_j
\end{equation*}
Let us admit this result for the moment, and assume therefore that the isotopy $\varphi_t$ fixes each non-contractible $L_j$ for all times $t\in [0,1 ]$.

Because the circles cannot intersect during this isotopy, and since $b$ does not depend on the particular isotopy one chooses without loss of generality (up to shrinking the tubular neighbourhoods we take off to produce $\tilde{\Sigma}_{g, p}$) the braid $b(\varphi)$ is represented by a braid entirely contained in $\tilde{\Sigma}_{g, p}$. All of this proves that the image of $b$ is contained in $\mathcal B_{k,0,p+2g}$, which is what we wanted. To prove equality, we just need to be able to represent all generators of $\mathcal{B}_{k,0,p+2g}$ as images of $b$. This is readily done: given any $\sigma_i$, we may consider a disc containing $L_1$ and $L_i$ and apply a half rotation. In the case of any $z_i$ (loop around a puncture) we may use an annulus with centre at the puncture, and a full rotation on it.

We may now end the proof showing that the generators of $\mathcal{B}_{\underline L}$, ${(\sigma_i)}_{i=1,\dots, k-1}$, ${(a_j)}_{j=1,\dots, g}$,  ${(b_j^{-1}a_jb_j)}_{j=1,\dots, g}$ and ${(z_l)}_{l=1,\dots, p-1}$ (which are elements of $\mathcal{B}_{k, g, p}$) are in fact the images of the generators of $\pi_1(\mathrm{Conf}^k(\tilde\Sigma_{g, p}))$ under the inclusion morphism. In fact, it is easy to see that the $\sigma_i$'s in $\mathcal{B}_{k, g, p}$ correspond to the half-twists in $\pi_1(\mathrm{Conf}^k(\tilde\Sigma_{g, p}))$, and that the turns around the punctures in $\tilde\Sigma_{g, p}$ are sent by the inclusion to turns around the punctures in $\Sigma_{g,p}$ (the $z_j$'s) and turns around the handles (the $a_i$'s and $b_i^{-1}a_ib_i$'s) (see Figure \ref{fig:b-1ab_generator} for a picture of $b_i^{-1}a_ib_i$).

    \begin{figure}
	\centering
	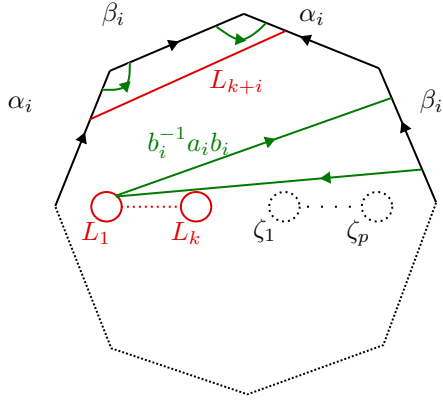
	\caption{The generator $b_i^{-1}a_ib_i$}
	\label{fig:b-1ab_generator}
    \end{figure}

\end{proof}
\begin{lem}\label{lem:loopHamDif}
Let $t\mapsto \varphi_t$ be an isotopy in $\Ham_c(\Sigma_{g, p})$ such that, for any $L_j$ non-contractible circle in a premonotone configuration $\underline L$ we have $\varphi_1(L_j)=L_j$. Then there is a loop in $\Ham_c(\Sigma_{g, p})$ based at the identity, $t\mapsto \psi_t$, such that for all $t$ and for all non contractible components $L_j$ one has
\begin{equation*}
    \psi_t\varphi_t(L_j)=L_j
\end{equation*}
\end{lem}
\begin{proof}
    Since $\varphi_1$ fixes all the non-contractible components $L_j$, we can find a Hamiltonian isotopy $(R_t)$, supported in tubular neighbourhoods of those components, fixing the $L_j$ for all $t$ and moreover such that $R_1\varphi^{-1}$ fixes pointwise small tubular neighbourhoods $U_j$ of the circles $L_j$. We remark that $R_1\varphi^{-1}$ defines a symplectic diffeomorphism $\mathfrak{R}$ of $\Sigma_{g,p}\setminus (\bigcup L_j)$ which is compactly supported in its interior. Remark that the union we are considering is only of the non-contractible $L_j$, and that we are not subtracting the contractible components of the Lagrangian link to $\Sigma_{g, p}$.
    
    We may choose $R$ so that $\mathfrak{R}$ is Hamiltonian. At first, we show that $\mathfrak{R}$ can be made isotopic via compactly-supported symplectomorphisms to the identity, and then we prove that we may in fact suppose it to be Hamiltonian. We use a slight variation of an argument from \cite{kha11}.
    \vspace{0.2 cm}

\textbf{Step 1}: $R$ may be chosen so that $\mathfrak{R}$ is symplectically isotopic to the identity.\vspace{0.2 cm}

\noindent
By the description given in \cite{farMar11} of $\pi_0(\Diff_0(\Sigma_{g,p}\setminus (\bigcup L_j)))$ and a Moser argument, the Symplectic Mapping Class Group of $\Sigma_{g,p}\setminus (\bigcup L_j)$ is generated by Dehn twists supported near the ends, since $\Sigma_{g,p}\setminus (\bigcup L_j)$ is a genus-0 surface. The symplectic diffeomorphism $\mathfrak{R}$ is isotopic to a non-trivial composition of Dehn twists (i.e. it is not isotopic to  the identity through compactly supported symplectomorphisms) if and only if there exist fixed points  of $R_1\varphi^{-1}$ in $U_j$ whose trajectories, along a Hamiltonian isotopy between $\mathrm{id}$ and $R_1\varphi^{-1}$ in $\Ham_c(\Sigma_{g, p})$, are not contractible. Remark that this condition does not depend on the choice of Hamiltonian isotopy by Lemma \ref{lem:HamSimplConn}. The homotopy type (more specifically, the number of rotations around a handle of the surface) of these trajectories determines the powers of the Dehn twists appearing in the isotopy class of $\mathfrak{R}$. Moreover, the homotopy type of such fixed points being constant on the whole of $U_j$ implies that the Dehn twists around the two ends corresponding to $L_j$ appear with opposite exponents. Consider a Hamiltonian $A$ whose time-1 map fixes pointwise the tubular neighbourhoods $U_j$, whose Hamiltonian vector field generates a flow that fixes $L_j$ (not necessarily pointwise) at all times, but such that the trajectories of the fixed points in $U_j$ have the same homotopy type as those of $R_1\varphi^{-1}$. In case $\Sigma_{g, p}$ is equipped with its standard symplectic structure, this just amounts to requiring $A$ on $U_j$ to have constant slope in the normal direction to $L_j$. We assume that $A$ is constant outside slightly larger tubular neigbourhoods of the $L_j$. The time-1 map defined by $A$ defines a symplectic diffeomorphism of $\Sigma_{g, p}\setminus \bigcup L_j$ whose isotopy class in the Symplectic Mapping Class Group is represented by the same product of Dehn twists as $\mathfrak{R}$. Since the Hamiltonian diffeomorphism of $\Sigma_{g, p}$ defined by $A$ is also supported in small tubular neighbourhoods of the $L_j$, we can replace $R_t$ with $(\varphi^t_A)^{-1} R_t$. It satisfies the same hypotheses as $R_t$, so that the resulting $\mathfrak{R}$ is now isotopic to the identity via compactly supported symplectomorphisms of $\Sigma_{g,p}\setminus (\bigcup L_j)$.\vspace{0.2 cm}

\textbf{Step 2}: $R$ may be chosen so that $\mathfrak{R}$ is Hamiltonian.

\vspace{0.2 cm}
\noindent
Consider an autonomous Hamiltonian $A_j$, supported in $U_j$, which is  constant and non-zero on a smaller neighbourhood of $L_j$. 
The Hamiltonian vector field of $A_j$ is by definition compactly supported inside $U_j$ but outside of a small neighbourhood of $L_j$, 
so that the time-1 map of $A_j$ defines a compactly supported symplectic diffeomorphism $\xi_j$ of $\Sigma_{g,p}\setminus \bigcup L_j$. Crucially, $\xi_j$ cannot be Hamiltonian as it has non-zero Flux. In Lemma \ref{lem:fluxNonZero} we are going to prove in fact that the $Flux(\xi_j)$ are linearly independent. Let us suppose this to be true for the moment.

We now claim that there are real numbers $c_1, \dots c_g$ such that
\begin{equation*}
    Flux\left(\mathfrak{R}\right)=\sum_{j=1}^gc_jFlux(\xi_j).
\end{equation*}
If the claim is verified, it suffices to compose $R_t$ with $(\varphi^t_{\sum_jc_jA_j})^{-1}$. The proof of this claim is included in Lemma \ref{lem:spanFluxXij}, since it uses the proof of Lemma \ref{lem:fluxNonZero}.
Therefore, we may assume that there exists a Hamiltonian $K$ supported away from the $L_j$ generating $R_1\varphi^{-1}$. Finally, we define $\psi_t=(\phi_K^t)^{-1}R_t(\varphi_{t})^{-1}$. Clearly, $\psi$ is a loop, and \[\psi_t\varphi_t(L_j)=(\phi_K^t)^{-1}R_t(\varphi_{t})^{-1}\varphi_t(L_j)=L_j\] since $R_t$ and $\phi_K^{t}$ fix $L_j$ at all time $t$.
\end{proof}

\begin{lem}\label{lem:fluxNonZero}
    The collection of the $Flux(\xi_j)$ in the proof of Lemma \ref{lem:loopHamDif} is linearly independent.
\end{lem}
\begin{proof}
    Let $\tilde{\omega}$ be the symplectic form induced on $\Sigma_{g, p}\setminus\bigcup L_j$ by the inclusion in $\Sigma_{g, p}$. Let $X_j$ be the Hamiltonian vector field of $A_j$, and $\tilde{X}_j$ be its restriction to $\Sigma_{g, p}\setminus\bigcup L_j$. By definition, we have the identities:
\begin{equation}\label{eq:equalityVectFields}
    \iota_{\tilde{X}_j}\tilde{\omega}=\iota_{X_j}\omega=-dA_j
\end{equation}
wherever both sides of the first equality are defined. Computing $Flux(\xi_j)$ (see \cite[Exercise 10.2.6]{mcDSal16}) requires us to integrate the 1-form $\iota_{\tilde{X}_j}\tilde{\omega}$ on the generators of 
\begin{equation*}
    H_1(\overline{\Sigma_{g, p}\setminus\bigcup L_j}, \partial(\overline{\Sigma_{g, p}\setminus\bigcup L_j}); \R).
\end{equation*}
The bar stands for the compactification obtained adding the boundary to every end. The surface $\overline{\Sigma_{g, p}\setminus\bigcup L_j}$ has then genus 0 and $2g+p$ boundary components. Each non-contractible $L_j$ we subtract gives rise to two boundary components. Recall that we are not subtracting the contractible components of the Lagrangian link.

A basis for the above homology group may for instance be constructed by fixing one boundary component of $\Sigma_{g, p}$, and taking a family of smooth curves with one endpoint on it, and the other one on a second boundary component of $\Sigma_{g, p}\setminus\bigcup L_j$. We assume that no two curves constructed this way have second endpoint on the same boundary component, which guarantees linear independence. Let $\gamma_i$ be such a curve, parametrised on the interval $[0, 1]$, connecting the fixed boundary component of $\Sigma_{g, p}$ and one of the two boundary components corresponding to $L_i$. Let $h_i$ be the homology class represented by $\gamma_i$. Then by Equation (\ref{eq:equalityVectFields}) above and the Fundamental Theorem of Calculus:
\begin{equation}\label{eq:fluxAj}
    Flux(\xi_j)(h_i)=\int_0^1\iota_{\tilde{X}_j}\tilde{\omega}(\dot{\gamma}_i(t))\, dt=-\int_0^1dA_j(\dot{\gamma}_i(t))\, dt=A_j(\gamma_i(0))-A_j(\gamma_i(1))
\end{equation}
The notation $\dot{\gamma}_i$ indicates the derivative of the path $\gamma_i$. Recall now that $A_j$ is non-zero only on $U_j$. We infer that:
\begin{equation}\label{eq:fluxAjDelta}
    Flux(\xi_j)(h_i)=-\delta_{i, j}A_j(L_j),
\end{equation}
where $\delta_{i, j}$ is Kronecker's delta and $A_j(L_j)$ is the value that $A_j$ takes on any point of the circle $L_j$. This proves that the collection of the $Flux(\xi_j)$ is linearly independent, and in particular that $\xi_j$ has non-zero $Flux$.
\end{proof}
\begin{lem}\label{lem:spanFluxXij}
    There exist real numbers $c_1, \dots, c_g$ such that \begin{equation*}
        Flux(\mathfrak{R})=\sum_{i=1}^gc_jFlux(\xi_j)
    \end{equation*}
\end{lem}

\begin{proof}
    We adopt the description from Lemma \ref{lem:fluxNonZero} of \begin{equation*}
        H_1(\overline{\Sigma_{g, p}\setminus\bigcup L_j}, \partial(\overline{\Sigma_{g, p}\setminus\bigcup L_j}); \R).
    \end{equation*}
The left-hand side of the equation is 0 whenever it is evaluated on the relative homology classes represented by smooth curves connecting two boundary components of $\Sigma_{g,p}$: such curves define homology classes in $H_1(\Sigma_{g, p}, \partial\Sigma_{g, p};\R)$, where $R$ is Hamiltonian. An application of Equation (\ref{eq:equalityVectFields}) with the fact that the $Flux$ of Hamiltonian isotopies is 0 concludes.
This is not yet enough to prove that such $c_j$ exist. Denote for each $j=1,\dots, g$ by $\gamma_{\pm, j}$ the two curves in the basis of the relative homology group with second endpoint on the two different boundary components associated to $L_j$. Let $h_{\pm, j}$ be the two associated homology classes, and $h^*_{\pm, j}$ the element of the associated dual basis for the cohomology. From Equation (\ref{eq:fluxAjDelta}) then we obtain
\begin{equation*}
    Flux(\xi_j)=-A_j(L_j)(h^*_{+, j}+h^*_{-, j}).
\end{equation*}
It is enough, to prove existence of the $c_j$, to show that $Flux(\mathfrak{R})$ takes the same value on $h_{+, j}$ and $h_{-, j}$. This is true: consider the concatenation $\gamma_{+, j}\#\overline{\gamma_{-,k}}$ (we assume without loss of generality that the endpoints coincide and that the concatenation is smooth), where the bar indicates inversion of time. It defines a homology class in $H_1(\Sigma_{g, p}, \partial\Sigma_{g, p};\R)$. Since $R_t$ is a Hamiltonian isotopy, so is the one denoted by $(R'_t)_{[0, 1]}$, defined by\begin{equation*}
    R'_t=\begin{cases}
        R_{2t}& t\in \left[0,\frac{1}{2}\right]\\
        R_{2t-1}\circ R_1& t\in \left[\frac{1}{2},1\right]
    \end{cases}.
\end{equation*}We have then
\begin{equation*}
    Flux(R'_t)(\gamma_{+, j}\#\overline{\gamma_{-,k}})=0.
\end{equation*}

On the other hand, if $H_t$ is a $1$-periodic Hamiltonian generating $R_t$, then $R'_t$ is generated by $H'_t=2H_{2t}$, and we have
\begin{align*}
    Flux(R'_t)&(\gamma_{+, j}\#\overline{\gamma_{-,k}})=-\int_0^1dH'_{t}\left(\frac {d}{dt}(\gamma_{+, j}\#\overline{\gamma_{-,k}})\right)\, dt\\
    &=-2\int_0^{\frac 1 2}dH_{2t}\left(2\dot{\gamma}_{+, j}(2t)\right)\, dt -2\int_{\frac 1 2}^1dH_{2t-1}\left(2\frac{d}{dt}\overline{\gamma_{-,k}}(2t-1)\right)\, dt\\
    &=-2\int_0^{1}dH_{t}\left(\dot{\gamma}_{+, j}(t)\right)\, dt -2\int_0^1dH_{t}\left(\frac{d}{dt}\overline{\gamma_{-,k}}(t)\right)\, dt\\
    &=2 (Flux(\mathfrak R)(h_{+,j})+Flux(\mathfrak R)(-h_{-,j}))
\end{align*}
which implies as we wanted to show that $Flux(\mathfrak R)(h_{+,j})$ and $Flux(\mathfrak R)(h_{-,j})$ coincide.

This finishes the proof of the Lemma.
\end{proof}

\begin{rem}
If $g=0, p=1$, Lemma \ref{lem:imb} proves surjectivity of the braid type function on the disc.
\end{rem}

\section{Proof of the main result}\label{sec:proofMainResult}
\subsection{Sketch of the proof of Theorem \ref{thm:braidPersistenceGenus} for the disc}\label{sec:proofResultDisc}
We are going to recall the main ideas behind the proof contained in \cite{M}. 

In the class of Hamiltonian diffeomorphisms $\Ham_{\underline L}(\Sigma_{g, p})$, the generators of $CF(\underline L, H)$ are essentially representatives of the braid class of $b(\varphi)$, and the reference path a trivial braid. If $\hat{y}$ is a capped Hamiltonian chord, the intersection product $[\hat{y}]\cdot\Delta$ can therefore be connected to the linking number of $b(\varphi)$: we aim to then isolate this contribution in the expression of the action.

We are lead to consider a difference of spectral invariants: we embed the disc $\D$ symplectically into spheres of areas $1+s_i$ ($i=1,2$), with $0\leq s_i\leq (k+1)A-1$, in such a way that the obtained Lagrangian links are monotone with same $\lambda$ but different $\eta$. We thus want to compute \begin{equation}\label{eq:diffSpecInv}
    c_{\underline{L}_{s_1}}(\varphi)-c_{\underline{L}_{s_2}}(\varphi)
\end{equation}
where the $\underline{L}_{s_i}$ are the two obtained Lagrangian links.

We define a biholomorphism between the symmetric products: this induces a chain isomorphism between the two $CF(\underline L_{s_i}, H)$, but such isomorphism does not respect the action filtration. The difference (\ref{eq:diffSpecInv}) is exactly the failure to respect the action filtration of this chain isomorphism. This means that the spectral invariants in (\ref{eq:diffSpecInv}) may be computed using capped Hamiltonian chords related by the biholomorphism. Since the two Lagrangian links share furthermore the same parameter $\lambda$, we may change the capping of this intersection point freely; without loss of generality therefore we may assume that such capping is contained in $\Sym^k\D\subset \Sym^k\sphere^2$. The proof ends showing that under this assumption
\begin{equation*}
    c_{\underline{L}_{s_1}}(\varphi)-c_{\underline{L}_{s_2}}(\varphi)= (\eta_{s_2}-\eta_{s_1})[\hat{y}]\cdot \Delta
\end{equation*}
where $\hat{y}$ is any capped Hamiltonian chord of $\Sym H$, showing that the intersection product coincides with the linking number up to a factor of $-2$. Choosing optimal $s_i$ and applying the Hofer Lipschitz property of the difference in (\ref{eq:diffSpecInv}) yields the main result of \cite{M}.
\subsection{Intersections in the symmetric product}\label{sec:intersections}
We define elementary intersections between two strands: they are local descriptions for intersections counting towards $u\cdot \Delta$ which do not admit a lift to the cartesian product. The goal of this section is to prove that elementary intersections are transverse in the symmetric product, and may thus be used to compute the intersection product with the diagonal.
\vspace{0.2 cm}

First off, let us reduce the computation of the intersection product to counting the intersections between two strands.

Let $u:[0, 1]\times [0, 1]\rightarrow \Sym^{k+g}(\Sigma_{g, p})$ be a capping for an intersection point between the $\Sym^{k+g}\underline L$ and $\Sym^{k+g}(\varphi(\underline L))$. Assume that for some $(s_0, t_0)$, $u(s_0, t_0)\in\Delta$. Generically, $u(s_0, t_0)$ belongs to the top stratum of the diagonal, which means that considering the path in the symmetric product
\begin{equation*}
    t\mapsto u(s_0, t)
\end{equation*}
as a collection of $k+g$ curves, at $t_0$ exactly two of them coincide. Furthermore, as we need to compute the sign of an intersection and this is a local problem, we may assume without loss of generality that the intersection at $(s_0, t_0)$ is the only intersection between $u$ and $\Delta$. Denote the two intersection paths in $u$ by $u_2:[0, 1]\times [0, 1]\rightarrow \Sym^2\Sigma_{g}$, and the 2-dimensional diagonal by $\Delta_2$.

\begin{lem}[Locality of the intersection problem]\label{lem:locality}
In the setting as above, the sign of the intersection at $(s_0, t_0)$ between $u$ and $\Delta$ is the same as the sign of the intersection at $(s_0, t_0)$ between $u_2$ and $\Delta_2$.
\end{lem}
\begin{proof}
Let us consider $u$ as above. Since $u$ intersects $\Delta$ exactly one and in the highest stratum, it factors through $\Sym^2 (\Sigma_g)\times \Sigma_g^{k+g-2}$ as in the following commutative diagram:
\[
\begin{tikzcd}
{[0,  1]^2} \arrow[r, "\tilde{u}"]\arrow[d, "u"]&\Sym^2 (\Sigma_g)\times \Sigma_g^{k+g-2}\arrow[dl, "p"]\\\
\Sym^g (\Sigma_g)&
\end{tikzcd}.
\]
In the diagram, $p$ represents the quotient projection. Now,  the first component of $\tilde{u}$ is exactly $u_2$, and by assumption
\begin{equation*}
    \tilde{u}\cdot(\Delta_2\times \Sigma_g^{k+g-2})=u_2\cdot \Delta_2.
\end{equation*}
Since the intersection pairing is functorial we have the equalities
\begin{equation*}
    u\cdot\Delta=(p\circ\tilde{u})\cdot\Delta=\tilde{u}\cdot(\Delta_2\times \Sigma_g^{k+g-2})=u_2\cdot \Delta_2
\end{equation*}
which prove what we wanted.

\end{proof}

\begin{defn}
    Let $(s, t)\in [0, 1]^2$. We say that a capping $u$ has an elementary intersection at $(s, t)$ with $\Delta$ if $u(s, t)\in \Delta$ and if there exist  charts as those defined in Lemma \ref{lem:locality} around $u(s, t)$ such that $u_2$ coincides with the roots of the complex polynomial
    \begin{equation*}
        X^2-(s+it)
    \end{equation*}
\end{defn}

\begin{lem}Elementary intersections between two strands are transverse.
\end{lem}

\begin{proof}We will start by considering a homotopy $\gamma : [-1,1]_s\times [-1,1]_t \rightarrow \Sym^2(\C)$ between two braids $\gamma_{-}$ and $\gamma_+$.
There is a diffeomorphism $\phi:\C^2\rightarrow \Sym^2(\C)$ which maps $(a,b)\in \C^2$ to the pair of roots of the degree 2 polynomial $X^2-aX+b$ (its inverse being given by $\phi^{-1}([x_-,x_+])=(x_-+x_+, x_-x_+)$).
Through this diffeomorphism, the diagonal $\Delta\subset \Sym^2(\C)$ corresponds to the set \[\phi^{-1}(\Delta)=\left\{(2x,x^2),x\in \C\right\}=\left\{\left(a,\frac {a^2}{4}\right),a\in \C\right\}\]whose tangent space at $(a,\frac {a^2}{4})$ is the complex vector subspace generated by $(1,\frac a 2)$, i.e. the real vector subspace generated by $(1,\frac a 2)$ and $(i,\frac {ia} 2)$.
Let $(a(s,t),b(s,t)):=\phi^{-1}(\gamma(s,t))$. Assume that $\gamma$ intersects the diagonal at $(s,t)=(0,0)$.
Then, the intersection is transverse if and only if the vectors $(1,\frac {a(0,0)} 2)$, $(i,\frac {ia(0,0)} 2)$, $(\partial_s a(0,0),\partial_s b(0,0))$ and $(\partial_t a(0,0),\partial_t b(0,0))$ generate the whole space $\C^2$ as a real vector space.

In particular, consider the case where $\gamma(s,t)$ is the pair of square roots of $s+it$. Then, $\gamma$ intersects the diagonal at $(0,0)$, and $(a(s,t),b(s,t))=(0,-s-it)$. Therefore, we get that the intersection is transverse if and only if the vectors $(1,0),(i,0),(0,-1),(0,-i)$ generate $\C^2$. Since this is true, we have transversality of the intersection.

\begin{figure}
        \centering
        \def\svgwidth{\columnwidth}
\begingroup%
  \makeatletter%
  \providecommand\color[2][]{%
    \errmessage{(Inkscape) Color is used for the text in Inkscape, but the package 'color.sty' is not loaded}%
    \renewcommand\color[2][]{}%
  }%
  \providecommand\transparent[1]{%
    \errmessage{(Inkscape) Transparency is used (non-zero) for the text in Inkscape, but the package 'transparent.sty' is not loaded}%
    \renewcommand\transparent[1]{}%
  }%
  \providecommand\rotatebox[2]{#2}%
  \newcommand*\fsize{\dimexpr\f@size pt\relax}%
  \newcommand*\lineheight[1]{\fontsize{\fsize}{#1\fsize}\selectfont}%
  \ifx\svgwidth\undefined%
    \setlength{\unitlength}{473.46696989bp}%
    \ifx\svgscale\undefined%
      \relax%
    \else%
      \setlength{\unitlength}{\unitlength * \real{\svgscale}}%
    \fi%
  \else%
    \setlength{\unitlength}{\svgwidth}%
  \fi%
  \global\let\svgwidth\undefined%
  \global\let\svgscale\undefined%
  \makeatother%
  \begin{picture}(1,0.32530459)%
    \lineheight{1}%
    \setlength\tabcolsep{0pt}%
    \put(0,0){\includegraphics[width=\unitlength,page=1]{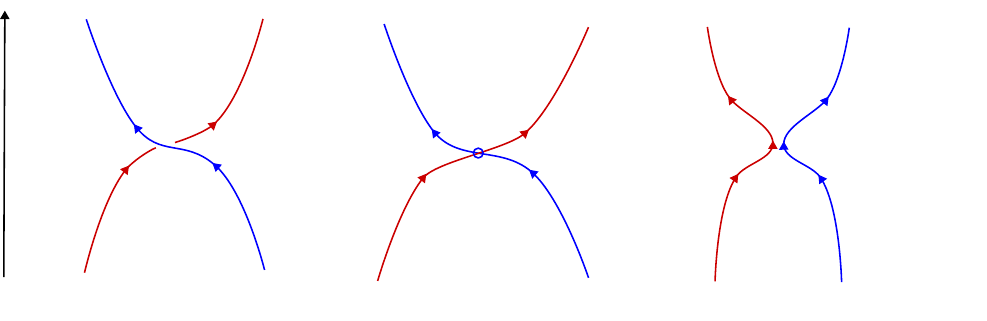}}%
    \put(0.01653124,0.30988351){\color[rgb]{0,0.03137255,0.03137255}\makebox(0,0)[lt]{\lineheight{1.25}\smash{\begin{tabular}[t]{l}$t$\end{tabular}}}}%
    \put(0.10903787,0.00471941){\color[rgb]{0,0,0}\makebox(0,0)[lt]{\lineheight{1.25}\smash{\begin{tabular}[t]{l}$t\mapsto \gamma(-0.01, t)$\end{tabular}}}}%
    \put(0.41206082,0.00471941){\color[rgb]{0,0,0}\makebox(0,0)[lt]{\lineheight{1.25}\smash{\begin{tabular}[t]{l}$t\mapsto \gamma(0, t)$\end{tabular}}}}%
    \put(0.72106498,0.00471941){\color[rgb]{0,0.03137255,0.03137255}\makebox(0,0)[lt]{\lineheight{1.25}\smash{\begin{tabular}[t]{l}$t\mapsto \gamma(0.01, t)$\end{tabular}}}}%
  \end{picture}%
\endgroup%

        \caption{The homotopy $\gamma(s,t)$}
\end{figure}
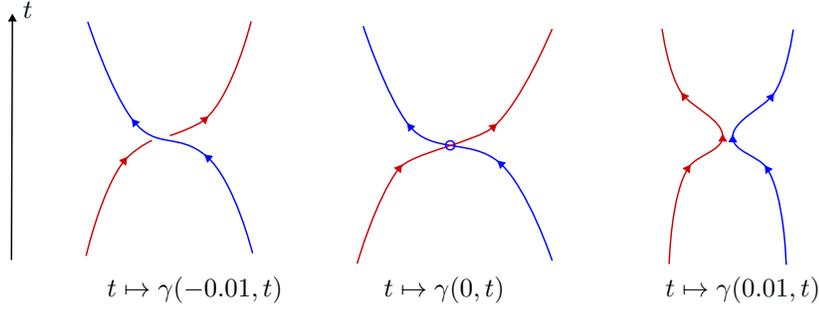
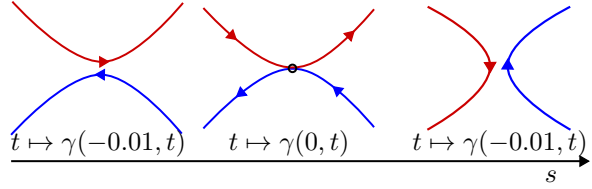
\begin{figure}
     \centering
\begingroup%
  \makeatletter%
  \providecommand\color[2][]{%
    \errmessage{(Inkscape) Color is used for the text in Inkscape, but the package 'color.sty' is not loaded}%
    \renewcommand\color[2][]{}%
  }%
  \providecommand\transparent[1]{%
    \errmessage{(Inkscape) Transparency is used (non-zero) for the text in Inkscape, but the package 'transparent.sty' is not loaded}%
    \renewcommand\transparent[1]{}%
  }%
  \providecommand\rotatebox[2]{#2}%
  \newcommand*\fsize{\dimexpr\f@size pt\relax}%
  \newcommand*\lineheight[1]{\fontsize{\fsize}{#1\fsize}\selectfont}%
  \ifx\svgwidth\undefined%
    \setlength{\unitlength}{257.50220832bp}%
    \ifx\svgscale\undefined%
      \relax%
    \else%
      \setlength{\unitlength}{\unitlength * \real{\svgscale}}%
    \fi%
  \else%
    \setlength{\unitlength}{\svgwidth}%
  \fi%
  \global\let\svgwidth\undefined%
  \global\let\svgscale\undefined%
  \makeatother%
  \begin{picture}(1,0.27195916)%
    \lineheight{1}%
    \setlength\tabcolsep{0pt}%
    \put(0,0){\includegraphics[width=\unitlength,page=1]{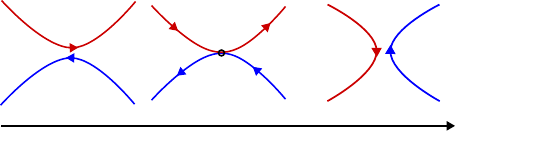}}%
    \put(0.78209045,0.00625034){\color[rgb]{0,0,0}\makebox(0,0)[lt]{\lineheight{1.25}\smash{\begin{tabular}[t]{l}$s$\end{tabular}}}}%
    \put(0.0007308,0.05350196){\color[rgb]{0,0,0}\makebox(0,0)[lt]{\lineheight{1.25}\smash{\begin{tabular}[t]{l}$t\mapsto \gamma(-0.01, t)$\end{tabular}}}}%
    \put(0.31950197,0.05350196){\color[rgb]{0,0,0}\makebox(0,0)[lt]{\lineheight{1.25}\smash{\begin{tabular}[t]{l}$t\mapsto \gamma(0, t)$\\\end{tabular}}}}%
    \put(0.58842924,0.05350196){\color[rgb]{0,0,0}\makebox(0,0)[lt]{\lineheight{1.25}\smash{\begin{tabular}[t]{l}$t\mapsto \gamma(-0.01, t)$\end{tabular}}}}%
  \end{picture}%
\endgroup%

     \caption{The homotopy $\gamma(s, t)$ as seen on $\C$.}
        \label{fig:elementary intersection}
\end{figure}
\end{proof}

\begin{rem}
From the proof we see that the sign of the intersection we studied above as local model is positive.
\end{rem}
\begin{rem}
The homotopy $s+it\mapsto \sqrt{(s+it)}\in \Sym^2\C$ does not lift to a function to $\C^2$. If we, in more generality, suppose that an intersection is transverse, the capping cannot be lifted to $\C^2$. Assume in fact that an intersection counting towards $[u]\cdot\Delta$ does lift to $\C^2$ (we take $k=2$ in light of Lemma \ref{lem:locality}). Then it cannot be transverse: up to time-translation, assume the intersection appears at $(s, t)=(0,0)$. Define
\begin{equation*}
    (\gamma_1, \gamma_2): (-\varepsilon, \varepsilon)\times(-\varepsilon, \varepsilon)\rightarrow \C^2
\end{equation*}
such that if $\pi: \C^2\rightarrow \Sym^2\C$ is the quotient projection, we locally have
\begin{equation*}
    \pi\circ (\gamma_1, \gamma_2)=u
\end{equation*}
Then $(\gamma_1, \gamma_2)(0, 0)\in \pi^{-1}(\Delta)$. The differential of $\pi$ being 0 there, so is the differential of $u$. This proves that the intersection of $u$ with $\Delta$ at $(0, 0)$ is not transverse.
\end{rem}

\subsection{Construction of the Hofer-Lipschitz function}\label{sec:constrFunction}

Let $\zeta_1,..., \zeta_p$ be the boundary components of $\Sigma_{g,p}$. Gluing discs of area $s_{i,j}$ along $\zeta_j$ for $i=1,2$, $j=1,...,p$ gives rise to two embeddings of $\Sigma_{g,p}$ into a closed surface of area $1+s_i$, $i=1,2$, denoted $\Sigma_g(1+s_i)$, where \begin{equation*}
    s_i=s_{i,1}+...+s_{i,p}\in(0, (k+1)A-1)
\end{equation*}
In the following, by $v_i$ we denote the vector
\begin{equation*}
    v_i:=(s_{i, 1}, \dots, s_{i, p})
\end{equation*}
and the set of possible values $v_1, v_2$ may take is
\begin{equation*}
    V:=\{(s_1,...,s_p)\in (\R_{\geq 0})^p|s_1+...+s_p\leq (k+1)A-1\}
\end{equation*}
Choose a diffeomorphism between the two closed surfaces: up to changing the complex structure on one of them it becomes a biholomorphism. We also assume this diffeomorphism to be the identity between small neighbourhoods of $\Sigma_{g, p}$ in $\Sigma_g(1+s_i)$.

This operation induces a bijective correspondence $\Phi$ between capping classes: it is the content of \cite[Lemma 3.2]{M} which still holds here (we need for this step the generalisation of monotonicity as proved in Section \ref{sec:monotonicity}). In our setting we may also prove, (see \cite[Lemmata 3.3, 3.4]{M} ), that gluing discs of these areas gives an isomorphism at the chain level of Floer complexes, and that this isomorphism commutes with the PSS maps. In \cite{M} the conventions are slightly different, and for the differential one counts holomorphic strips, while here the strips have to solve the Floer equation. This does not represent a problem in our context: since the Hamiltonian is 0 on the discs we glue, the biholomorphism between surfaces $\Sigma_g(1+s_i)$ still defines a chain map since it sends the Hamiltonian vector field on $\Sigma_g(1+s_1)$ to the one on $\Sigma_g(1+s_2)$.

We still need to compute the action shift induced by the change in the areas: we do this in the following lemma. Before stating it, we recall the basic notations and notions the computation involves.

Let $\varphi\in\Ham_{\underline L}(\Sigma_{g,p})$, and let $H$ be a Hamiltonian generating $\varphi$, supported in the interior of $\Sigma_{g,p}$. Let $\underline L_{v_i}$ be the image of $\underline L$ inside $\Sigma_g(1+s_i)$, $i=1,2$, and $H_{v_i}$ the extension by 0 of $H$ in the closed surfaces.
After a small perturbation of $H$, we can assume that the Heegaard-Floer complexes for $(\underline L_{v_i},H_{v_i})$, $i=1,2$ are well defined.
Using the bijective correspondence between cappings above, $[\hat y_1]$ corresponds to a generator $[\hat y_2]:=\Phi([\hat y_1])$ of $CF^*(\Sym^{k+g}H_{v_2},\Sym\underline L_{v_2})$.

\begin{lem}
\label{lem:action_difference}
    The difference of action $\mathcal A_{H_{v_1}}([\hat y_1])-\mathcal A_{H_{v_2}}([\hat y_2])$ does not depend on the choice of the generator $[\hat y_1]$.
\end{lem}

\begin{proof}
    Start by fixing a generator $[\hat y_1]$.
    Since $H_{v_i}$, $i=1,2$ are supported in $\Sigma_{g,p}$, we get that $y_i$ are paths in $\Sym^{k+g}(\Sigma_{g,p})$. The biholomorphism between $\Sigma_g(1+s_1)$ and $\Sigma_g(1+s_2)$ restricts to the identity on $\Sigma_{g,p}$, so the two paths $y_1$ and $y_2$ coincide, and $H_{v_1}|_{y_1}=H_{v_2}|_{y_2}=H|_{y_1}$.
    Therefore, using the formula for the action, the terms containing the integral of the Hamiltonian will cancel out, leaving us with:
    \[\mathcal A_{H_{v_1}}([\hat y_1])-\mathcal A_{H_{v_2}}([\hat y_2])=\eta_{s_2}[\hat y_2]\cdot \Delta - \eta_{s_1}[\hat y_1]\cdot \Delta + \omega_{s_2}([\hat y_2])-\omega_{s_1}([\hat y_1])\]
    We first show that this quantity does not depend on the choice of capping for the path $y_1$. Let $[\hat y'_1]$ be another capping. We want to show that \begin{align*}
     \eta_{s_2}([\hat y_2]\#\Phi[\hat y'_1]^{-1})\cdot \Delta + &\omega_{s_2}([\hat y_2]\#\Phi[\hat y'_1]^{-1})-\\-&\left(\eta_{s_1}([\hat y_1]\#[\hat y'_1]^{-1})\cdot \Delta + \omega_{s_1}([\hat y_1]\#[\hat y'_1]^{-1})\right)=0   
    \end{align*}
    Since $[\hat y_1]\#[\hat y'_1]^{-1}$ and $\Phi[\hat y_1]\#\Phi[\hat y'_1]^{-1}$  are homotopies between $x$ and itself, it is in fact a disc with boundary on $\Sym\underline L$, so by monotonicity (Proposition \ref{prp:monotonicity}), \begin{align*}
        &\eta_{s_1}([\hat y_1]\#[\hat y'_1]^{-1})\cdot \Delta + \omega_{s_1}([\hat y_1]\#[\hat y'_1]^{-1})=\frac \lambda 2 \mu([\hat y_1]\#[\hat y'_1]^{-1})\\
        \eta_{s_2}(&\Phi[\hat y_1]\#\Phi[\hat y'_1]^{-1})\cdot \Delta + \omega_{s_2}(\Phi[\hat y_1]\#\Phi[\hat y'_1]^{-1})=\frac \lambda 2 \mu(\Phi[\hat y_1]\#\Phi[\hat y'_1]^{-1})
    \end{align*}Since the Maslov index is preserved by the bijective correspondence, the difference above vanishes, which proves that the difference of action does not depend on the choice of capping.
    Now we show that it does not depend on the path $y_1$ either. If $y'_1$ is another trajectory of $H_{v_1}$ between $\Sym\underline L_{v_1}$ and itself, then we can define a capping for $y'_1$ by choosing any capping for $y_1$ and concatenating it with any homotopy $[w_1]$ from $y_1$ to $y'_1$. Since we showed that the action difference does not depend on the capping, it is enough to find a single homotopy $[w_1]$ for which \begin{equation*}
        \eta_{s_2}(\Phi[w_1])\cdot \Delta + \omega_{s_2}(\Phi[w_1])-\left(\eta_{s_1}([w_1])\cdot \Delta + \omega_{s_1}([w_1])\right)=0
    \end{equation*}
    Such a homotopy exists: we can first find a homotopy sliding the end points of $y_1$ and $y_1'$ inside the circles so that they coincide (and therefore define braids), and then we can take a homotopy inside $\Sym^{k+g}(\Sigma_{g,p})$ that does not intersect $\Delta$, because $y_1$ and $y'_1$ have the same braid type.
\end{proof}

The previous Lemma implies that for any $\varphi\in\Ham_{\underline L}(\Sigma_{g,p})$, and any choice of generator $[\hat y_1]$, we have $c_{\underline L_{v_1}}(\varphi)-c_{\underline L_{v_2}}(\varphi)=\frac 1 {k+g}(\mathcal A_{H_{v_1}}([\hat y_1])-\mathcal A_{H_{v_2}}([\hat y_2]))$ as in \cite[Lemma 3.6]{M}. \begin{defn}
    We define $f_{v_1,v_2}:\Ham_c(\Sigma_{g, p})\rightarrow \R$ as
    \begin{equation*}
        f_{v_1,v_2}(\varphi):=c_{\underline L_{v_1}}(\varphi)-c_{\underline L_{v_2}}(\varphi)=\frac 1 {k+g}(\mathcal A_{H_{v_1}}([\hat y_1])-\mathcal A_{H_{v_2}}([\hat y_2]))
    \end{equation*}
\end{defn}

Moreover, we have:

\begin{lem} The map $f_{v_1,v_2}:\Ham_{\underline L}(\Sigma_{g,p})\to\R$ is a group homomorphism, which factorises over $b:\Ham_{\underline L}(\Sigma_{g,p})\to\mathcal B_{\underline L}$, i.e. there exists a group homomorphism $\mathfrak f_{v_1,v_2}:\mathcal B_{\underline L}\to\R$ such that $f_{v_1,v_2}= \mathfrak f_{v_1,v_2}\circ b$.
\end{lem}

\begin{proof}
    Let $\varphi$, $\psi$ be in $\Ham_{\underline L}(\Sigma_{g,p})$. Let $H$ (resp. $H'$) be a Hamiltonian supported inside $\Sigma_{g,p}$ generating $\varphi$ (resp. $\psi$). Let $[\hat y_1]$ (resp. $[\hat y'_1]$) be a generator of $CF^*(\Sym^{k+g}H_{v_1},\Sym\underline L_{v_1})$ (resp. $CF^*(\Sym^{k+g}H'_{v_1},\Sym\underline L_{v_1})$.
Let $z:[0,1]\to \Sym^{k+g}\Sigma_{g,p}$ be a trajectory of $H\#H'$ from $\Sym\underline L$ to itself. We want to construct a capping $\hat {z}_1: [0,1]\times [0,1]\to \Sym^{k+g}\Sigma_{g}(1+s_1)$ from $z$ to the constant path $x$, using the cappings $\hat y_1$ and $\hat y'_1$.
We define it as follows (see Fig. \ref{fig:capping_^z}):

    \begin{figure}
	\centering
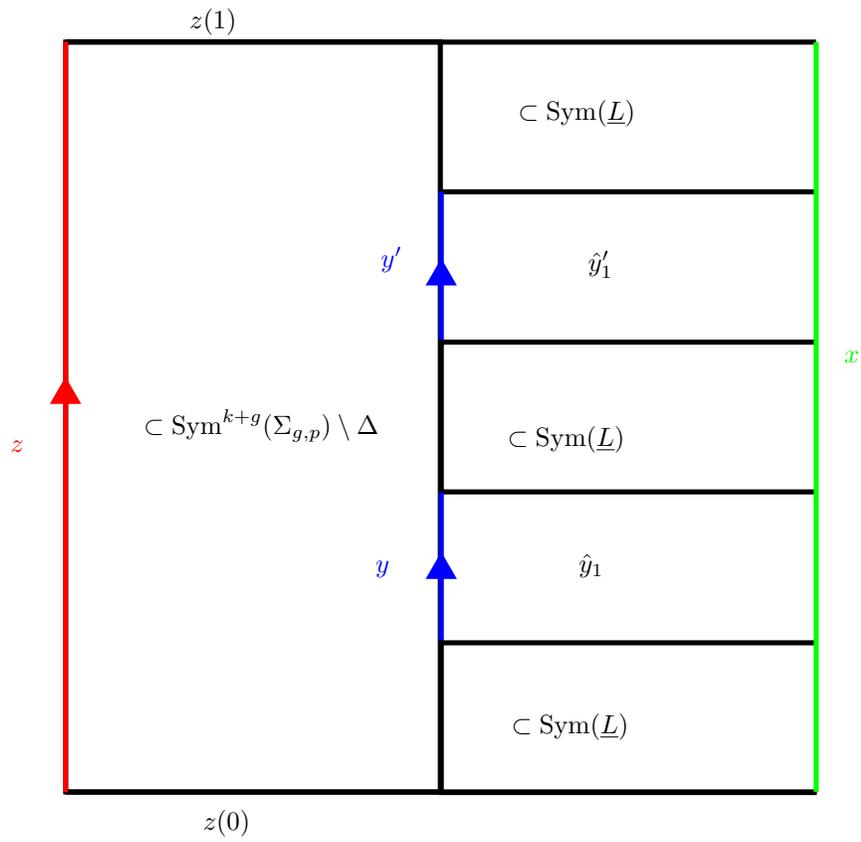
	\caption{The capping $\hat z_1$}
	\label{fig:capping_^z}
    \end{figure}

\begin{itemize}
    \item On $[\frac 1 2,1]\times [0,\frac 1 5]$:
    \begin{itemize}
        \item $\hat {z}_1([\frac 1 2,1]\times [0,\frac 1 5])\subset \Sym \underline L_{v_1}$
        \item for $0\leq t \leq \frac 1 5$, $\hat {z}_1(1,t)=x$
        \item $\hat {z}_1(\frac 1 2,0) = z(0)$
        \item for $\frac 1 2\leq s \leq 1$, $\hat {z}_1(s,\frac 1 5)=\hat {y}_1(2s-1,0)$
    \end{itemize}
    (Such a homotopy exists because $\Sym\underline L_{v_1}$ is path-connected, it is in fact isotopic to a Clifford torus, cf. \cite{CGHMSS22})
    \item On $[\frac 1 2,1]\times [\frac 1 5,\frac 2 5]$, $\hat {z}_1(s,t)=\hat {y}_1(2s-1,5t-1)$
    \item On $[\frac 1 2,1]\times [\frac 2 5,\frac 3 5]$:
    \begin{itemize}
        \item $\hat {z}_1([\frac 1 2,1]\times [\frac 2 5,\frac 3 5])\subset \Sym \underline L_{v_1}$
        \item for $\frac 2 5\leq t \leq \frac 3 5$, $\hat {z}_1(1,t)=x$
        \item for $\frac 1 2 \leq s \leq 1$, $\hat {z}_1(s,\frac 2 5)=\hat {y}_1(2s-1,1)$
        \item for $\frac 1 2 \leq s \leq 1$, $\hat {z}_1(s,\frac 3 5)=\hat y'_1(2s-1,0)$
    \end{itemize}
    \item On $[\frac 1 2,1]\times [\frac 3 5,\frac 4 5]$, $\hat {z}_1(s,t)=\hat y'_1(2s-1,5t-3)$
    \item On $[\frac 1 2,1]\times [\frac 4 5,1]$:
    \begin{itemize}
        \item $\hat {z}_1([\frac 1 2,1]\times [\frac 4 5,1])\subset \Sym \underline L_{v_1}$
        \item for $\frac 4 5\leq t \leq 1$, $\hat {z}_1(1,t)=x$
        \item for $\frac 1 2\leq s \leq 1$, $\hat {z}_1(s,\frac 4 5)=\hat y'_1(2s-1,1)$
        \item $\hat {z}_1(\frac 1 2,1) = z(1)$
    \end{itemize}
    \item On $[0,\frac 1 2]\times [0,1]$, $\hat {z}_1$ is a homotopy of braids between $z$ and $\hat {z}_1(\frac 1 2,\cdot)$, contained in $\Sym^{k+g}(\Sigma_{g,p})$. Indeed, since $b$ is a group homomorphism, those two paths are isotopic as braids.
\end{itemize}

Now, the capping $\hat {z}_1$ consists of a concatenation of $\hat {y}_1$, $\hat y'_1$, homotopies contained in $\Sym\underline L_{v_1}$, and a braid isotopy contained in $\Sym^{k+g}(\Sigma_{g,p})$. Homotopies contained in $\Sym^{k+g}(\Sigma_{g,p})$ do not contribute to the difference of symplectic area, and as $\Sym\underline L_{v_1}$ is away from the diagonal, and (by definition) a braid isotopy does not cross the diagonal, the only contribution to the difference of action comes from $\hat {y}_1$ and $\hat y'_1$. Since the intersection number and the symplectic area are additive, we get that $f_{v_1,v_2}(\varphi\psi) = f_{v_1,v_2}(\varphi)+f_{v_1,v_2}(\psi)$. Moreover, since braid isotopies do not contribute to the action difference, $f_{v_1,v_2}(\varphi)$ only depends on the braid type of $\varphi$.
\end{proof}

Therefore, it is enough to compute the value of $\mathfrak f_{v_1,v_2}$ on the generators of $\mathcal B_{\underline L}$ to express $f_{v_1,v_2}(\varphi)$ (granted that we know how to decompose the braid type of $\varphi$ as a product of generators).

\begin{lem}
    The values of $\mathfrak f_{v_1,v_2}$ are the following:\begin{align*}
    \mathfrak f_{v_1,v_2}(a_i)=2\frac {\eta_2-\eta_1}{k+g}=-f_{v_1,v_2}(b_i^{-1}a_ib_i),\,\, \,\mathfrak f_{v_1,v_2}(\sigma_j)=-\frac {\eta_2-\eta_1}{k+g},\\\mathfrak f_{v_1,v_2}(z_j)=\frac {s_{2,j}-s_{1,j}}{k+g}=-2(\eta_2-\eta_1)\frac{k+2g-1}{k+g}\frac {s_{2,j}-s_{1,j}}{s_2-s_1}
    \end{align*}
\end{lem}

\begin{proof}
    To do these computations, according to Lemma \ref{lem:action_difference} it is enough to produce explicit homotopies between the reference path $x$ (which is a trivial braid) and the braid for which we know how to compute the action difference. For $a_i$, $b_i^{-1}a_ib_i$ and $\sigma_j$, we exhibit homotopies contained in $\Sigma_{g,p}$, so that the only contribution to the action difference comes from intersections with the diagonal. We also choose our homotopies so that the intersections with the diagonal are transverse. To be sure of their transversality, we make use the content of Section \ref{sec:intersections}. We are going to show there that the intersections we count here are transverse, and give their signs.
    
    \begin{figure}
	\centering
	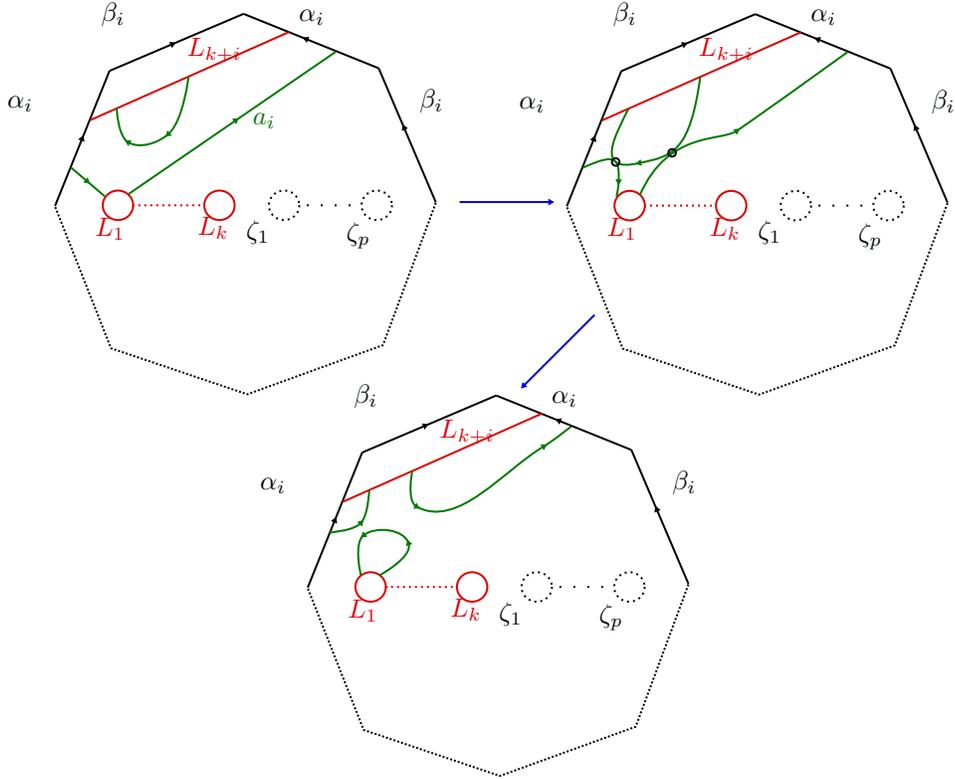
	\caption{Homotopy between $a_i$ and the constant path}
	\label{fig:a_homotopy}
    \end{figure}

    \begin{figure}
	\centering
	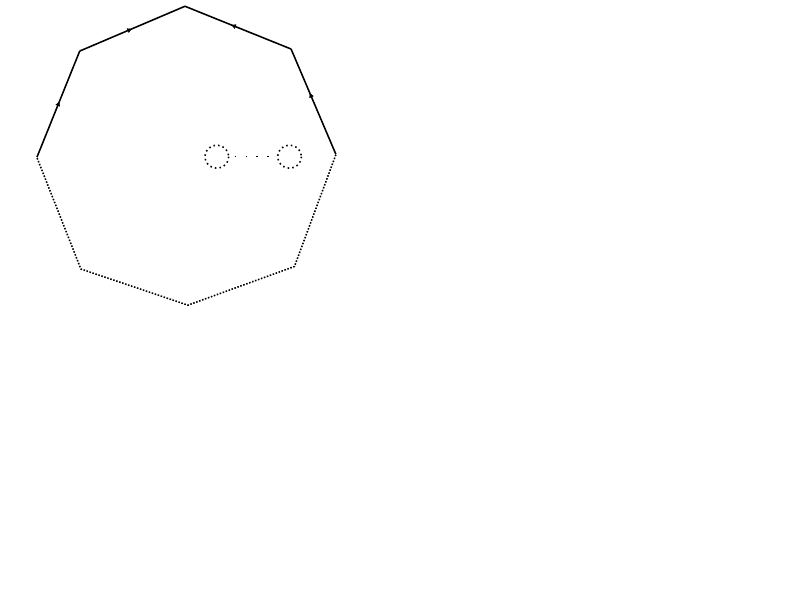
	\caption{Homotopy between $b_i^{-1}a_ib_i$ and the constant path}
	\label{fig:b-1ab_homotopy}
    \end{figure}

    \begin{figure}
	\centering
	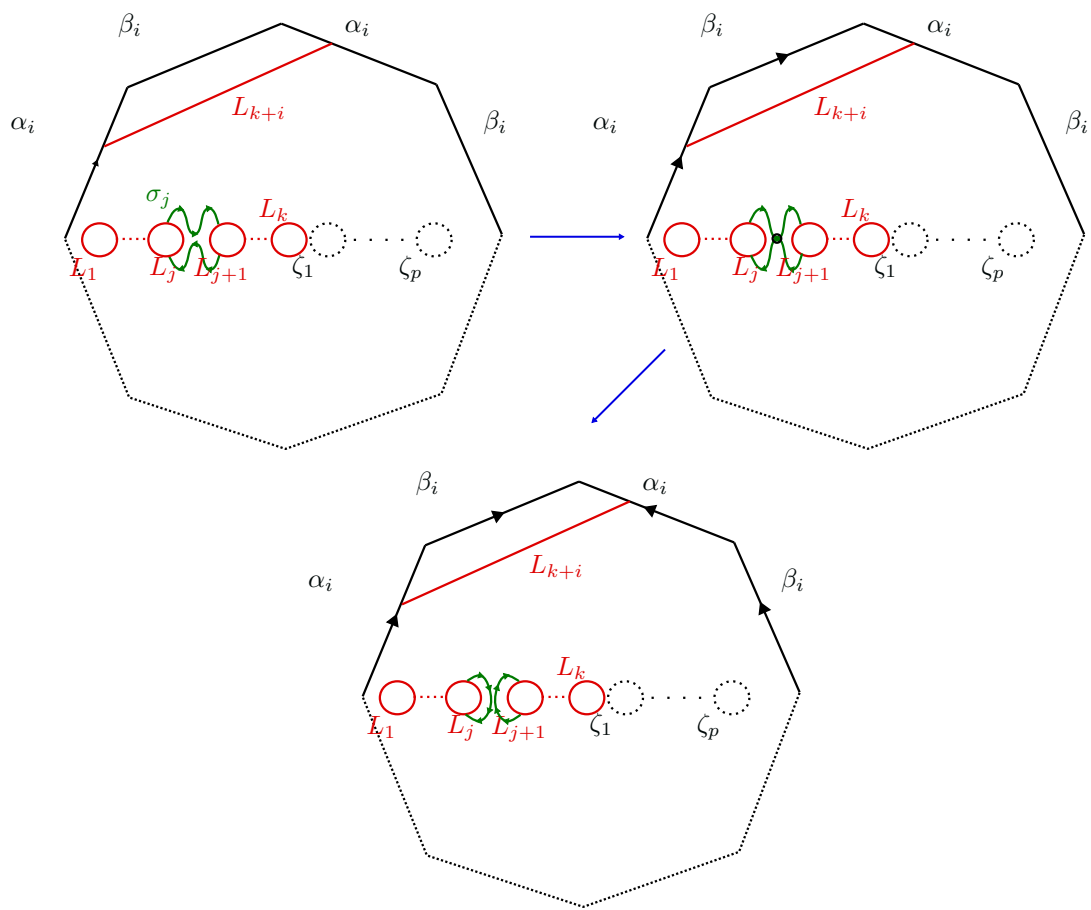
	\caption{Homotopy between $\sigma_j$ and the constant path}
	\label{fig:sigma_homotopy}
    \end{figure}

    The result comes from counting such intersections (with sign) on Figures \ref{fig:a_homotopy}, \ref{fig:b-1ab_homotopy} and \ref{fig:sigma_homotopy}. All the intersections appearing in those homotopies are modelled on the example of Figure \ref{fig:elementary intersection}, therefore we know they are transverse and can compute their signs.

    For $z_j$, the braid becomes trivial after embedding $\Sigma_{g, p}$ into $\Sigma_g(1+s_i)$, and we can choose an isotopy to the trivial braid (without crossing) which sweeps the disc glued to the boundary component $\zeta_j$. Therefore, the only contribution to the action difference comes from the difference of symplectic area between the two embeddings:
    \[\mathfrak f_{v_1,v_2}(z_j)=\frac {s_{2,j}-s_{1,j}}{k+g}\]
    By the monotonicity property applied to the only non-contractible connected component of $\Sigma_g(1+s_i)$, we have
    \[A = 1+s_i - kA + 2\eta_{i}(k+2g-1)\]
    and therefore $\eta_{2}-\eta_{1}=\frac {s_1 - s_2}{2(k+2g-1)}$, which implies that
    \[\mathfrak f_{v_1,v_2}(z_j)=-2(\eta_2-\eta_1)\frac{k+2g-1}{k+g}\frac {s_{2,j}-s_{1,j}}{s_2-s_1}\]
\end{proof}
\begin{rem}
    With this last expression for $\mathfrak f_{v_1,v_2}(z_j)$, it is easy to check that the values of $\mathfrak f_{v_1,v_2}$ are consistent with the relation of Remark \ref{rem:relation}.
\end{rem}
\begin{rem}
The reader might be surprised by the fact that \begin{equation*}
    \mathfrak{f}_{(v_1, v_2)}(b_i a_i b_i^{-1})\neq \mathfrak{f}_{(v_1, v_2)}(a_i)
\end{equation*}Now, remark that the $\mathfrak f_{(v_1, v_2)}$ are not defined on the $b_i$, which indeed are not elements in $\mathcal{B}_{\underline L}$.
\end{rem}

By the Hofer-Lipschitz property of link spectral invariants, we have that for all $\varphi\in\Ham_{\underline L}(\Sigma_{g,p})$, and for any choice of $(v_1,v_2) = (s_{i,j})_{i=1,2;1\leq j \leq p}\in V\times V$:

\[\norm{\varphi}\geq\frac 1 2 \abs{\mathfrak f_{v_1,v_2}(\varphi)}\]

To get the best estimate, we compute the maximum of $\abs{\mathfrak f_{(s_{i,j})}(\varphi)}$ over the choice of $(v_1,v_2)$.
We have 
\[(k+g)\mathfrak f_{v_1,v_2}(\varphi)=\frac {s_1-s_2}{2(k+2g-1)}(2k_{gen}-k_{\sigma})+\sum\limits_{j=1}^{p-1}k_j(s_{2,j}-s_{1,j})\]
where we have decomposed $b(\varphi)$ as a product of the generators $a_i$, $(b_i^{-1}a_ib_i)^{-1}$, $\sigma_j$ and $z_j$, and $k_{gen}$ is the sum of the exponents of all the $a_i$ and $(b_i^{-1}a_ib_i)^{-1}$ in this decomposition; $k_{\sigma}$ is the sum of the exponents of all the $\sigma_j$, and $k_j$ is the sum of the exponents of $z_j$.

\begin{lem}
    $k_{gen}$, $k_\sigma$ and the $k_j$ do not depend on the decomposition of $b(\varphi)$ as a product of generators.
\end{lem}

\begin{proof}
    $k_{gen}$, $k_\sigma$ and the $k_j$ can be seen as well-defined group homomorphisms $\mathcal B_{\underline L}\to\R$. Indeed, one can define them on the free group generated by the $a_j$'s, $b_j^{-1}a_jb_j$'s, $\sigma_i$'s and $z_l$'s, and check that they vanish on the set of relations defining $\mathcal B_{\underline L}$ (see \cite[Theorem 1.1]{bel04} for the full list of relations). Therefore, they descend to group homomorphisms $\mathcal B_{\underline L}\to\R$ and hence do not depend on the choice of decomposition of $b(\varphi)$ as a product of generators.
\end{proof}

Observe that this expression is homogeneous in $(v_1,v_2)$, i.e. $|\mathfrak f_{\kappa v_1,\kappa v_2}(\varphi)|=\abs{\kappa}\abs{\mathfrak f_{v_1,v_2}(\varphi)}$. Therefore, the maximum has to be attained when $s_1$ or $s_2$ is maximal, i.e. equal to $(k+1)A-1$. Since permuting $v_1$ and $v_2$ does not change $\abs{\mathfrak f_{v_1,v_2}(\varphi)}$, we can assume that $0\leq s_1\leq s_2=(k+1)A-1$. Let $k_{max}:= \max\{k_j\}$, $k_{min}:=\min\{k_j\}$, and let $j_{max}$ and $j_{min}$ be indices such that $k_{j_{max}}=k_{max}$, and $k_{j_{min}}=k_{min}$. Then,

\begin{align*}
    (k+g)\mathfrak f_{v_1,v_2}(\varphi)\leq \frac {s_1-s_2}{2(k+2g-1)}(2k_{gen}-k_{\sigma})+\sum\limits_{s_{2,j}-s_{1,j}\geq 0}k_{max}(s_{2,j}-s_{1,j})+\\+\sum\limits_{s_{2,j}-s_{1,j}<0}k_{min}(s_{2,j}-s_{1,j})
    \leq \frac {s_1-s_2}{2(k+2g-1)}(2k_{gen}-k_{\sigma})+s_2 k_{max} - s_1 k_{min}\\
    = s_2\left(k_{max}-\frac {2k_{gen}-k_{\sigma}}{2(k+2g-1)}\right)+s_1\left(\frac {2k_{gen}-k_{\sigma}}{2(k+2g-1)}-k_{min}\right)
\end{align*}
with equality when $s_{1,j}=\delta_{j,j_{min}}s_1$ and $s_{2,j}=\delta_{j,j_{max}}s_2$.

Similarly,
\begin{align*}
    (k+g)\mathfrak f_{v_1,v_2}(\varphi)\geq \frac {s_1-s_2}{2(k+2g-1)}(2k_{gen}-k_{\sigma})+\sum\limits_{s_{2,j}-s_{1,j}\geq 0}k_{min}(s_{2,j}-s_{1,j})+\\+\sum\limits_{s_{2,j}-s_{1,j}<0}k_{max}(s_{2,j}-s_{1,j})\geq \frac {s_1-s_2}{2(k+2g-1)}(2k_{gen}-k_{\sigma})+s_2 k_{min} - s_1 k_{max}\\
    = s_2\left(k_{min}-\frac {2k_{gen}-k_{\sigma}}{2(k+2g-1)}\right)+s_1\left(\frac {2k_{gen}-k_{\sigma}}{2(k+2g-1)}-k_{max}\right)
\end{align*}
with equality when $s_{1,j}=\delta_{j,j_{max}}s_1$ and $s_{2,j}=\delta_{j,j_{min}}s_2$.
Since those expressions are linear in $s_1$, the extremal values are attained for $s_1=0$ or $s_1=s_2$. Set:
\begin{equation*}
    R= k_{max}-k_{min},\, S=k_{max}-\frac {2k_{gen}-k_{\sigma}}{2(k+2g-1)},\, T=\frac {2k_{gen}-k_{\sigma}}{2(k+2g-1)}-k_{min} 
\end{equation*}
In the end, we get:
\begin{align}\label{eqn:f}\max\limits_{(v_1,v_2)\in V^2}\abs{\mathfrak f_{v_1,v_2}(\varphi)}=\frac{(k+1)A-1}{k+g}\max\left\{R, S, T \right\}
\end{align}
and therefore
\begin{align*}
    \norm{\varphi}\geq \frac {(k+1)A-1}{2(k+g)}\max\left\{R, S, T \right\}
\end{align*}

\section{Hofer norms for  braid groups on surfaces with boundary}\label{sec:chenAdapt}

Chen, in \cite{chen23}, proved non-degeneracy of the pseudonorms defined by the first author in \cite{M}. We aim to show how Chen's proof may be adapted to extend his theorem to:
\begin{thm}\label{thm:non-degeneracy}
    Let $\underline{L}$ be a pre-monotone link on $\Sigma_{g, p}$, a compact symplectic surface of genus $g$ and with $p$ boundary components. Then there exists an $\varepsilon>0$, only depending on $\underline{L}$ , such that if $\varphi, \psi\in\Ham_{\underline{L}}(\Sigma_{g, p})$ are such that $d_H(\varphi, \psi)<\varepsilon$, then $b(\varphi)=b(\psi)$.
\end{thm}
This theorem in turn implies that the pseudonorms we defined are non-degenerate.
\begin{cor}
    If $(\varphi_i)\subset \Ham_{\underline{L}}(\Sigma_{g, p})$ is a sequence of Hamiltonian diffeomorphisms preserving $\underline{L}$, and $\norm{\varphi_i}\to 0$, then there exists a positive integer $n$ such that for every $i>n$ the braid $b(\varphi_i)$ is trivial.
\end{cor}
We are now going to adapt Chen's proof to our setting.

In \cite{chen23}, the author considers Hofer-close Hamiltonian diffeomorphisms $\varphi$ and $\psi$. Choosing generating Hamiltonians for both of them and perturbing them gives rise to the associated Floer complexes, and continuation maps between them. If the generating Hamiltonians are $\mathcal{C}^0$-close, one may furthermore prove that the pseudo-holomorphic curves appearing in the continuation maps do not intersect the diagonal of the symmetric product, nor the divisor $D=z+\Sym^{k-1}(\sphere^2)$, where $z$ is the north pole of the sphere (outside the image of the embedding $\D\hookrightarrow \sphere^2$). The fact that the two Hamiltonians may be chosen to be close of course follows from the definition of the Hofer norm. One may then use these pseudo-holmorphic curves to produce a braid isotopy between chosen representatives of $b(\varphi)$ and $b(\psi)$ using the pseudo-holomorphic curves of the continuation maps.

In our case we wish to prevent pseudo-holomorphic curves from intersecting the diagonal and exiting the surface $\Sigma_{g, p}$: we glue discs of appropriate areas to the boundary components, so that $\underline{L}$ becomes monotone, and consider the centres $\zeta_i$, $i=1,\dots, p$ of the discs we just glued. Define the divisors $D_i=\zeta_i+\Sym^{k+g_1}(\Sigma_g)$: if $u:\D\rightarrow \Sym^k(\Sigma_g)$ is a pseudoholomorphic disc with Lagrangian boundary conditions on the link $\underline{L}$, if $[u]\cdot D_i=0$ for each $i$ then the image does not leave the image of the embedding $\Sigma_{g, p}\hookrightarrow \Sigma_g$, by positivity of the intersections between (pseudo)holomorphic submanifolds.

As in \cite{chen23}, we then decompose the differential in the Floer complex as a sum of contributions, each of them counting pseudo-holomorphic discs with fixed intersection number with the diagonal $\Delta$ and the divisors $D_i$. The function $\partial_{00}$, counting pseudo-holomorphic discs not intersecting $\Delta$ or any of the divisors $D_i$ is a differential, and we consider the complex $(CF(\underline{L}, H), \partial_{00})$. As Chen points out, it is possible to do the same with continuation maps: if $K$ is a compactly supported Hamiltonian generating $\psi\in \Ham_{\underline{L}}(\Sigma_{g, p})$, let $h: CF(\underline{L}, H)\rightarrow CF(\underline{L}, K)$ be the continuation map associated to a regular homotopy between $H$ and $K$. One may consider $h_{00}$, obtained from $h$ counting contribution of pseudo-holomorphic curves not intersecting $\Delta$ or any $D_i$: it turns out to be a chain homotopy  between the two complexes $(CF(\underline{L}, H), \partial_{00})$ and $(CF(\underline{L}, K), \partial_{00})$. The arguments used in \cite{chen23} to prove existence of pseudo-holomorphic curves contributing to $h_{00}$ (assuming that $\norm{H-K}_{(1, \infty)}$ is small) carry over to our case: they in fact either involve local considerations around the Lagrangians, or use monotonicity of the link which we have now proved (see Proposition \ref{prp:monotonicity}). As a consequence, one may follow through the proof of Theorem 1 in \cite{chen23}, which provides the braid isotopy we are looking for: such an isotopy is given by gluing a pseudo-holomorphic curve contributing to $h_{00}$ with a disc whose image does not leave the torus $\Sym^{k+g}\underline{L}$. This completes the proof of Theorem \ref{thm:non-degeneracy}.

\section{A proof using Künneth formula}
\label{sec:proof_Kunneth}
In this section we are going to present how, assuming the Hamiltonian diffeomorphisms at play are supported in a disc contained in the surface, one can prove Theorem \ref{thm:braidPersistenceGenus} combining directly the results from \cite{M} and \cite{MT}. The former reference contains a statement about the Hofer geometry on the disc, and the latter one shows how to deduce results on the Hofer geometry of a surface with genus and boundary.

Assume a disc $D$ embedded in $\Sigma_{g,p}$, ($g, p\geq 1$) contains all the contractible components of a pre-monotone link $\underline L$, and does not intersect the non-contractible ones.

Denote by $\Ham_{\underline L, D}(\Sigma_{g,p})$ the set of Hamiltonian diffeomorphisms $\varphi$ satisfying:
\begin{itemize}
    \item $\varphi(\underline L) = \underline L$;
    \item the support of $\varphi$ is contained in $D$.
\end{itemize}

We have a map $\Ham_{\underline L, D}(\Sigma_{g,p})\rightarrow \mathcal{B}_{k, 0,1}$ which is a surjective morphism of groups. 

\begin{thm}
\label{thm:inequality}
Let $\varphi\in\Ham_{\underline L, D}(\Sigma_{g,p})$; denote by $h$ its braid type. Then,
\[\norm{\varphi}\geq \frac 1 {4(k+g)}\frac {(k+1)A-1}{(k+2g-1)}\abs{\mathrm{lk}(h)}\]
\end{thm}

\begin{proof}
    For $s\geq 0$, if $\Sigma_g$ has non-empty boundary, we choose a symplectic embedding of $(\Sigma_{g,p},\omega)$ into a closed surface $(\Sigma_g(s), \omega^s)$ of genus $g$ and area $1+s$, so that $\underline L$ becomes an $\eta_s$-monotone link, where $$\eta_s = \frac {(k+1)A-(1+s)}{2(k+2g-1)}$$
    We get a Hofer-Lipschitz spectral invariant $c^s_{\underline L}$, defined for a Hamiltonian $H:S^1\times \Sigma_g(s)\to\R$ by
    \[c^s_{\underline L}(H)=\frac{1}{k+g}c^s_{\Sym \underline L}(\Sym^{k+g} H)\]
    where $c^s_{\Sym \underline L}$ is the Lagrangian spectral invariant of $\Sym \underline L\subset \Sym^{k+g}(\Sigma_g(s))$.
    We write $\underline L = L_1\cup...\cup L_k\cup \alpha_1\cup...\cup \alpha_g$, where the $\alpha_i$'s are the non-contractible components, and define $\underline K:= L_1\cup...\cup L_k \subset D$.
    One can view $\Sigma_g(s)$ as a connected sum $\sphere^2\#E_1\#...\#E_g$, where the sphere $\sphere^2$ contains $D$ (and therefore $\underline K$), and the $E_i$ are copies of the 2-dimensional torus, with $\alpha_i\subset E_i$.
    Now, suppose $H$ is supported in $D$, and denote by $\underline L'$ (resp. $\underline K'$) the image of $\underline L$ (resp. $\underline K$) by $\phi^1_H$. By taking a small perturbation of $H$ sending the $\alpha_i$ to some transverse $\alpha_i'$, and applying $g$ times the K\"unneth formula \cite[Theorem 6]{MT}, we get an isomorphism of filtered chain complexes
    \[CF^*(\Sym \underline L,\Sym \underline L')\cong CF^*(\Sym \underline K, \Sym \underline K')\otimes \bigotimes\limits_{1\leq i \leq g} CF^*(\alpha_i,\alpha_i')\]
   On the right hand side, $\underline K$ and $\underline K'$ are links in $(\sphere^2,\omega^s_{\sphere^2})$, and the $\alpha_i, \alpha_i'$ are circles in $(E_i,\omega^s_i)$, where the symplectic forms satisfy the following:
   \begin{itemize}
       \item $\omega^s_{\sphere^2}$ coincides with $\omega^s$ on $D$ (which is away from the connected sum region), and has total area $1+s+4g\eta_s$;
       \item $\omega^s_i$ coincides with $\omega^s$ on a neighbourhood of the homotopy between $\alpha_i$ and $\alpha_i'$.
   \end{itemize}
   Hence, we get that $c^s_{\Sym \underline L}(\Sym^{k+g} H)=c^s_{\Sym \underline K}(\Sym^{k} H)$, and therefore
   \[c^s_{\underline L}(H)=\frac k {k+g} c^s_{\underline K}(H)\]
    where the right hand side is computed in $(\sphere^2,\omega^s_{\sphere^2})$.
    Since this holds for any $H$ supported in $D$, it holds for a Hamiltonian generating $\varphi$ and for its iterates, so we get after homogenisation (and renormalisation) \begin{align*}
        \mu^s_{\underline L}(\varphi):=\lim\limits_{n\to\infty}\left(\frac 1 n c^s_{\underline L}(H^{\#n})-n\int_{S^1}\int_{\Sigma_g,p}H_t\omega dt\right)&=\\=\frac {k}{k+g}\lim\limits_{n\to\infty}\frac 1 n c^s_{\underline K}(H^{\#n})-\int_{S^1}\int_{\sphere^2}H_t\omega dt=\frac k {k+g} &\mu^s_{\underline K}(\varphi)
    \end{align*}
    (it is shown in \cite[Section 7.3]{CGHMSS22} that the (renormalised) homogenised spectral invariants do not depend on the choice of $H$ generating $\varphi$).
    By the Hofer-Lipschitz property of spectral invariants, we obtain
    \[2\norm{\varphi}\geq\abs{ \mu^0_{\underline L}(\varphi)-\mu^{(k+1)A-1}_{\underline L}(\varphi)}=\frac k {k+g} \left(\abs{\mu^0_{\underline K}(\varphi)-\mu^{(k+1)A-1}_{\underline K}(\varphi)}\right)\]
    Now, according to \cite[Theorem 1.4]{M}, \[\abs{\mu^0_{\underline K}(\varphi)-\mu^{(k+1)A-1}_{\underline K}(\varphi)}=\frac 1 k \eta_0\abs{\mathrm{lk}(h)}=\frac 1 {2k}\frac {(k+1)A-1}{(k+2g-1)}\abs{\mathrm{lk}(h)}\]and therefore we get the theorem. 
\end{proof}

\printbibliography
\end{document}